\documentclass[11pt,a4paper]{article}

\usepackage{mathtools}
\usepackage{authblk} 
\usepackage{algpseudocode,algorithmicx,algorithm}

\usepackage{mathrsfs}
\usepackage{latexsym,bm}
\usepackage{amsmath,amsfonts,amssymb,amsthm,rotating}
\usepackage{extarrows}

\usepackage{graphicx,subfigure,epstopdf,float}
\usepackage{enumerate,cases,multirow}
\usepackage{makecell}
\usepackage{caption}

\usepackage{longtable,arydshln,threeparttable}
\usepackage{indentfirst}
\usepackage[top=25mm,bottom=20mm,left=25mm,right=20mm]{geometry}
\usepackage{cite}

\usepackage{listings}

\usepackage{makeidx}        
\usepackage{booktabs}
\usepackage[bookmarksnumbered,hyperindex,linktocpage=true,hidelinks]{hyperref}

\newcommand{\ket}[1]{| #1 \rangle} 
\newcommand{\bra}[1]{\langle #1 |} 

\newcommand{\bb}{\boldsymbol}

\def \d {\mathrm{d}}
\def \e {\mathrm{e}}
\def \i {\mathrm{i}}

\newcounter{parentalgorithm}

\makeatother

\newtheorem{theorem}{Theorem}[section]
\newtheorem{lemma}{Lemma}[section]

\newtheorem{definition}{Definition}[section]

\theoremstyle{remark}
\newtheorem{remark}{\bf Remark}[section]

\numberwithin{equation}{section}

\usepackage{qcircuit}

\begin{document}

\title{ Quantum preconditioning method for { finite difference discretizations of the Poisson equation} via Schr\"odingerization}

\author[1]{Shi Jin\thanks{shijin-m@sjtu.edu.cn}}
\author[1, 2]{Nana Liu\thanks{nana.liu@quantumlah.org} }
\author[3]{Chuwen Ma\thanks{cwma@math.ecnu.edu.cn}}
\author[4,5,6]{Yue Yu\thanks{terenceyuyue@xtu.edu.cn}}
\affil[1]{School of Mathematical Sciences, Institute of Natural Sciences, MOE-LSC, Shanghai Jiao Tong University, Shanghai, 200240, China}
\affil[2]{Global College, Shanghai Jiao Tong University, Shanghai 200240, China}
\affil[3]{School of Mathematical Sciences, Key Laboratory of MEA, Ministry of Education, Shanghai Key Laboratory of PMMP, East China Normal University, Shanghai 200241, China,
	}
\affil[4]{School of Mathematics and Computational Science, Xiangtan University, Xiangtan, Hunan 411105, China}
\affil[5]{Hunan Research Center of the Basic Discipline Fundamental Algorithmic Theory and Novel Computational Methods, Xiangtan, Hunan 411105, China}
\affil[6]{National Center for Applied Mathematics in Hunan, Xiangtan, Hunan 411105, China}

\maketitle

\begin{abstract}
We present a quantum preconditioning framework for solving linear systems { arising from a finite difference discretization of the Poisson equation}. It is based on the combination of the Schr\"odingerization technique \cite{JLY22b,JLYPRL24} and the BPX multilevel preconditioner in order to achieve near-optimal complexity. The Schr\"odingerization technique transforms linear partial and ordinary differential equations into Schr\"odinger-type systems with unitary evolution in one higher dimension,  making them suitable for quantum simulation.
{ A key contribution is a structure-aware construction of the block-encoding for the symmetrically preconditioned  matrix $A_S = S^\top A S$, where $A$ is the stiffness matrix and $S$ encodes the BPX preconditioner in factored form. By establishing a novel commuting identity, we avoid the unfavorable normalization scaling that would otherwise arise from naive multiplication of block-encodings. This yields an exact block-encoding of $A_S$ with normalization $\mathcal{O}(d^2(L+1))$, where $d$ is the spatial dimension and $L$ is the number of levels. Combined with the Schr\"odingerization-based Hamiltonian simulation, the overall quantum algorithm achieves a query complexity of $\mathcal{O}\big(\mathrm{poly}(d)\varepsilon^{-1} \mathrm{polylog}(\varepsilon^{-1}) \big)$ for estimating linear functionals of the solution to a given tolerance $\varepsilon$.}
\end{abstract}

\textbf{Keywords}: Linear systems, Schr\"odingerization, BPX preconditioner, Block-encoding

\textbf{MSC codes}: 68Q12, 65F10, 65F08

 \tableofcontents

\section{Introduction}
Quantum computing has emerged as a promising paradigm to address the inherent
limitations of classical computing, which, after decades of rapid advancement,
is nearing its physical limits \cite{Fre69,DBKFPTZ22}. One of the most
significant applications of this quantum advantage lies in the efficient
resolution of linear systems
\cite{Ambainis12,CKS17,HHL09,LT20,SSO19,WZP18,XSEBY21,TAWL21}. By enabling
faster and more efficient solutions to linear systems, quantum computing has the
potential to revolutionize numerical simulations, with profound implications for
fields such as physics and engineering.

In this paper, we focus on an efficient quantum solver for linear systems of the form
\begin{equation} \label{eq:linear_algebra}
A \bm{x} = \bm{b},
\end{equation}
{ which arise from the finite difference discretization of the Poisson equation, with $A$ being symmetric positive definite.}
Here $A\in \mathbb{R}^{N\times N}$ is a sparse matrix and $\bm{b}\in
\mathbb{R}^N$ is a vector. Without loss of generality, we assume that the
dimension satisfies $N = 2^n$.

Quantum computing offers a promising approach to addressing the computational
challenges associated with matrix dimensionality, requiring only a logarithmic
number of qubits with respect to the matrix dimension to store both the matrix
and the solution vector. Over the past decade, several quantum algorithms have
been developed to enhance the performance of generic quantum linear system
problem (QLSP) solvers \cite{CKS17,HHL09,LT20,SSO19,gilyen2019quantum}.
Using the standard amplitude amplification technique \cite{BHM02}, the query
complexity of the original Harrow-Hassidim-Lloyd (HHL) quantum linear systems
algorithm \cite{HHL09} scales as ${ \mathcal{O}}(\kappa_A^2\varepsilon^{-1} )$,
where $\varepsilon$ is the target accuracy,
$\kappa_A = \|A\|\|A^{-1}\|$ is the condition number.  { Unless otherwise stated, the term ``query complexity'' refers specifically to the number of queries to the matrix input models (i.e., the block-encoding oracles for the relevant matrices), and does not include the additional sampling overhead required for readout.}
Subsequently, several advanced techniques, such as the linear combination of
unitaries (LCU) \cite{CKS17}, quantum singular value transformation (QSVT)
\cite{gilyen2019quantum}, and quantum signal processing (QSP)
\cite{Low2019qubitization},
have been proposed to improve the dependence on the condition number and reduce
the query complexity to  ${\mathcal{O}}(\kappa_A^2
\text{polylog}(\kappa_A\varepsilon^{-1}) )$. In \cite{CKS17}, variable-time
amplitude amplification (VTAA) was combined with a low-precision phase
estimation approach, improving the complexity of the LCU method to
${ \mathcal{O}}(\kappa_A \text{polylog}(\kappa_A\varepsilon^{-1}))$, which is nearly
optimal with respect to both $\kappa_A$ and $\varepsilon$.
An alternative approach is the randomized method (RM) \cite{SSO19}, which, when
combined with the fast eigenpath traversal method \cite{BKS09}, achieves a query
complexity that is also nearly optimal at $\mathcal{O}(\kappa_A
\text{polylog}(\kappa_A\varepsilon^{-1}))$.
However, both VTAA and RM rely on highly intricate processes that pose
significant challenges for practical implementation.

Despite substantial advancements in quantum computing, the computational
complexity of solving quantum linear systems remains strongly influenced by the
condition number. Indeed, the numerical approximations to linear partial differential equations often result in ill-conditioned linear systems. For example, when solving the Poisson equation, the condition number is proportional to the square of the number of grid points. To reduce the computational resources for both classical and quantum algorithms, it is important to develop preconditioners to reduce the condition number.   Addressing large condition numbers through preconditioning for quantum algorithms has
been explored in prior works, including \cite{CJS13, SX18, TAWL21, BNWA23}.

Among the notable strategies for solving the quantum linear problem is
Schr\"odingerization \cite{JLY22b,JLYPRL24}, which transforms any linear
partial differential equations (PDEs) or ordinary differential equations (ODEs)
into a higher-dimensional Schr\"odinger-type equation.
In \cite{JL22}, the Schr\"odingerization is applied to linear systems based on the Jacobi and Power Iterative methods.

Inspired by the ideas in \cite{JL22,JLY22b,JLYPRL24}, this paper extends  Schr\"odingerization to all convergent linear stationary iterative algorithms,
described by the general form:
\begin{equation}\label{eq:sta ite}
	\bm{x}^{\text{new}} = \bm{x}^{\text{old}} + B(\bm{b} - A\bm{x}^{\text{old}}),
\end{equation}
where $B$ is referred to as the iterator, which can also serve as a preconditioner. The convergence of such iterative methods corresponds to the steady state of the associated continuous system.
This steady state can subsequently be obtained  by  quantum  {\it Hamiltonian simulation} within the Schr\"odingerization framework.
The quantum Jacobi iteration proposed in \cite{JL22} is a special case, where the iterator operator is defined as the inverse of the diagonal matrix of $A$.

To prepare the solution state $\ket{\bm{x}}$ within accuracy $\varepsilon$ with constant success probability, the complexity estimate in our framework depends on the block-encoding normalization factor $\alpha_{BA}$ and the reciprocal lower spectral bound $1/\lambda_{\min}(BA)$ of the preconditioned operator $BA$. { This shows that, in the quantum setting, the effectiveness of a preconditioner is determined not only by the spectral improvement of the system, but also by whether the preconditioned operator admits an efficient block-encoding with a favorable normalization factor.
For the Poisson equation, the BPX multilevel preconditioner provides uniform spectral bounds for the preconditioned operator $BA$ in the classical theory. However, this alone does not guarantee an efficient quantum implementation, since the block-encoding cost may still be large.} To make the preconditioner genuinely effective for our quantum algorithm, we exploit the factorization $B=SS^\top$ with $S=[S_0,S_1,\cdots,S_L]$ and work with the symmetrized matrix $A_S=S^\top A S$. { This strategy was first adopted in \cite{DP24}, where the BPX preconditioner was used to transform the linear system arising from finite element discretizations into a well-conditioned form more amenable to quantum computation.}


{
A naive approach would block-encode the preconditioning factor $S$ and the matrix  $A$ separately and then multiply them to obtain $A_S$. However, this leads to a block-encoding normalization factor of order $\|S\|^2\|A\|$, which is essentially $\kappa(A)$ since $S\approx A^{-1/2}$. To avoid this unfavorable dependence, similar to the idea in \cite{DP24}, we construct a block-encoding of $A_S$ directly rather than through a product of separate encodings.
The key observation is based on the factorization
$ A=C_L^\top C_L$, with $C_L = h^{(d-2)/2}D_L$,
where $h$ is the mesh size on the finest level and $D_L$ is the forward discrete gradient operator. Thus, $C_L$ is the weighted discrete gradient associated with the discrete energy. Similarly, $D_\ell$ denotes the corresponding forward discrete gradient operator on level $\ell$.
Since $A_S=S^\top A S=(C_LS)^\top(C_LS)$, and $S\approx A^{-1/2}$, the operator $C_LS$ can be viewed as a balanced discrete gradient whose spectral norm is mesh-uniform up to constants. To make this structure explicit, we prove the commuting relation
$ C_LS_\ell = T_{\ell\to L}D_\ell $
in Theorem~\ref{thm:CS=TD},
where $T_{\ell\to L}$ is a transfer operator from the coarse level $\ell$ to the finest level $L$ for the discrete gradient. 
This representation rewrites $C_LS_\ell$ as a combination of a coarse-level discrete gradient operator and its corresponding transfer operator.
}

{ Our approach to constructing the commuting identity differs from that in \cite{DP24}. In their work, the preconditioning operation $C_L S_\ell$ is transformed into a preconditioner in the discontinuous Galerkin space. By choosing an orthogonal basis for the discontinuous Galerkin space, they obtain $T_{\ell \to L}$ with orthogonal columns, thereby transferring the entire block-encoding constant of $C_L S_\ell$ to $D_\ell$. Extending this approach to the finite difference setting~--~by interpreting nodal values as finite element degrees of freedom~--~is possible but somewhat complicated.
In our treatment, we note that $S_\ell$ is a prolongation from level $\ell$ to level $L$, and $C_L$ corresponds to differentiation at level $L$. The linearity of the prolongation allows us to first consider differentiation at level $\ell$ and then apply the prolongation from level $\ell$ to level $L$. We show that the resulting $T_{\ell \to L}$ also has orthogonal columns  (up to a constant factor), which similarly enables the full transfer of the block-encoding constant of $C_L S_\ell$ to $D_\ell$.}

{
In developing our method, we highlight two main improvements over the BPX-based quantum realization in \cite{DP24}.}

\begin{enumerate}
  \item { \textbf{Simpler commuting identity construction.}}
     The construction of the commuting identity is simpler, and its block-encoding implementation is also more straightforward. In \cite{DP24}, the realization is relatively complicated due to the transition between continuous finite element and discontinuous Galerkin basis functions, as well as the associated ordering issues. Our construction, based directly on the action  of $C_L S_\ell$, is more natural and avoids such intricate transfer operations and ordering complications.

  \item { \textbf{Structure-aware block-encoding in a finite-difference setting.}}
  Because we work on nested dyadic Cartesian grids, the discrete operators in our scheme
have explicit tensor-product forms with fixed local stencils. This allows us to construct
the required block-encodings directly from a few simple one-dimensional building blocks,
and then lift them to multiple dimensions by tensor-product rules. As a result, the
implementation is more transparent and avoids the elementwise assembly, basis-transfer,
and ordering complications that appear in more general finite-element constructions.

\end{enumerate}

	The structure of the paper is as follows:
    In Section \ref{sec:SchrPre}, we apply Schr\"odingerization to the ODE system based on classical linear iterative methods to derive the Hamiltonian system for quantum computing. In Section \ref{sec:BPX}, we briefly review the BPX preconditioner for a finite difference discretization of the Poisson equation on nested dyadic grids, and provide explicit forms of the stiffness matrix and the preconditioner. Section \ref{sec:BE} addresses the block-encoding of the symmetrically preconditioned  matrix $A_S = S^\top A S$, where we establish the key commuting identity and present the explicit block-encoding of $A_S$.  In Section \ref{sec:implementation}, we present the implementation of the Schr\"odingerization-based Hamiltonian simulation for the preconditioned system, along with the computation of a linear quantity and its complexity. Finally, conclusions are presented in Section \ref{sec:con}.

We introduce some notations used throughout the paper.
The relations $f_1 \lesssim g_1$ and $f_2 \gtrsim g_2$ mean that $f_1 \leq C_1 g_1$ and $f_2 \geq C_2 g_2$, respectively, where the constants $C_1$ and $C_2$ are independent of the discretization parameters. Scalar-valued quantities are denoted by standard symbols, while vector-valued quantities are written in boldface. We adopt a 0-based indexing, i.e., $j \in {0,1,\cdots,N-1}$ (or equivalently $j \in [N]$). We use sans-serif letters such as $\mathsf{a},\mathsf{f}$, and $\mathsf{b}$ to denote quantum registers. The identity matrix and the null matrix are denoted by $I$ and $O$, respectively; their dimensions are usually clear from the context, and in case of ambiguity, we write $I_n$ or $I_{N\times N}$ for the $N=2^n$-dimensional identity matrix.

\section{Schr\"odingerization for symmetrically preconditioned systems} \label{sec:SchrPre}
Building on the ideas from \cite{JL22, JLY22b, JLYPRL24}, we first reformulate the convergent stationary iterative algorithm as a linear ODE system, where the solution to the algebraic problem corresponds to the steady state of the ODE. This reformulation enables the application of Schr\"odingerization to construct a Hamiltonian system for quantum computing.

\subsection{A review of the iteration for the preconditioned linear system problems}
\label{sec:itera and pre}
A linear stationary iterative method can be expressed in the following general form:
\begin{algorithm}[H]
\caption{For any matrix $B\in \mathbb{R}^{N\times N}$, $\bm{x}^{\text{new}} = ITER(\bm{x}^{\text{old}})$}\label{alg:SIM}
		\begin{algorithmic}
			\State 1. Form residual: $\bm{r} = \bm{b} - A \bm{x}$.
			\State 2. Solve error equation: $\hat{\bm{e}} = B \bm{r}$.
			\State 3. Correct iteration: $\bm{x}^{\text{new}} = \bm{x}^{\text{old}}+\hat{\bm{e}}$.
		\end{algorithmic}
\end{algorithm}
Thus, the new iteration is obtained by computing \eqref{eq:sta ite}.
Here $B$ is called the {\it iterator}.
Without loss of generality, we assume that $B$ is symmetric.
Typical examples include the Richardson iteration, where $B = \omega I$ with $\omega>0$ as a constant, and the Jacobi iteration, where $B = D^{-1}$ with $D$ denoting the diagonal of $A$.
It is noticed that each iteration only depends on the previous approximate solution $\bm{x}^{\text{old}}$, and the iterative method performs the same operations in each iteration. This characteristic is the reason it is referred to as a stationary (or explicit) iterative algorithm.

We denote $\sigma(A)$ the set of eigenvalues of $A$, and define the spectral radius $\rho(A):=\sup\{|\lambda|:\lambda\in \sigma(A)\}$.
For the stationary algorithm, we have a well-known convergence result.
\begin{lemma}\cite{XJ92}
	The stationary iterative algorithm \ref{alg:SIM} converges for any initial guess if and only if the spectral radius $\rho(I - BA)<1$.
\end{lemma}

It is noted that the core of the iterative algorithm is the operator $B$ which captures the essential information of $A^{-1}$.  The approximate inverse $B$, can be used as a {\it preconditioner}.
Instead of computing \eqref{eq:linear_algebra}, one computes
\begin{equation}\label{preconditioningLSP}
	B A \bm{x} = B \bm{b}.
\end{equation}
A good preconditioner, improves the convergence of the iterative method sufficiently and is relatively inexpensive to compute.
However, it is noted that the linear iterator $B$ may not be convergent at all, whereas $B$ can always be a preconditioner. For example, the Jacobi method is not convergent for all symmetric and positive definite, but the diagonal part for the subspace equation can always be used as a preconditioner in the multigrid and multilevel methods, which is often known as the diagonal preconditioner. To get a convergence iteration, one can replace $B$ by $\omega B$ for $\omega \in (0, 2/\rho(BA))$.
\subsection{Steady-state solution of linear ODEs}

According to the stationary iterative algorithm \ref{alg:SIM}, 	we can define an ODE system
\begin{equation}\label{eq:ODE}
	\frac{{\rm d} \bm{u}(t)}{{\rm d} t}  =
	-B A \bm{u}(t) +B \bm{b},\quad
	\bm{u} (0) = \bm{u}_0,
\end{equation}
where $\bm{u}_0$ is an initial guess.
When both $A$ and $B$ are symmetric positive definite, we can introduce the eigenvalues of $BA$ as follows. Let the eigenvector of $BA$ be $\bm{w}_j$, with the corresponding eigenvalue $ \lambda_j $, i.e.,
		\[
		BA \bm{w}_j = \lambda_j \bm{w}_j \quad \text{or} \quad A \bm{w}_j = \lambda_j B^{-1} \bm{w}_j, \quad j = 1, \cdots, N.
		\]
		The problem on the right-hand side is called the generalized eigenvalue problem, and it has a complete orthogonality system with respect to $ B^{-1} $.
        Additionally, let $W=[\bm{w}_1,\cdots,\bm{w}_N]$
        denote the matrix of eigenvectors used in the diagonalization of $BA$.

       The relation between the stationary iterative algorithm and the steady-state solution to \eqref{eq:ODE} is described below.
\begin{lemma}\label{lem:rel stationary}
	If the stationary iterative algorithm \ref{alg:SIM} converges, which implies $\lambda(BA) \in (0,2)$, then the ODE system $\eqref{eq:ODE}$ has a steady-state solution, denoted by $\bm{u}_{\infty} =\lim\limits_{t\to+ \infty}\bm{u}(t)$, satisfying $\bm{x} = \bm{u}_{\infty}$ and
	\[
	\|\bm{u}(t)-\bm{x}\| \le  \kappa(W)   \e^{-\lambda_{\min}(BA) t} \|\bm{u}(0) - \bm{x}\|,
	\]
	where $\kappa(W)$ is the condition number of the transition matrix $W$ for the diagonalization of $BA$.
\end{lemma}

	As observed in the above lemma, the solution to the linear systems problem can be recast as the steady-state solution to the ODE system \eqref{eq:ODE}. For this reason, we can develop a quantum algorithm by using the Schr\"odingerization approach in \cite{JLY22b,JLYPRL24} for linear ODE systems, which transforms linear PDEs and ODEs with non-unitary dynamics into Schr\"odinger-type systems via the so-called warped phase transformation that maps the equation into one higher dimension by introducing an auxiliary variable.

{ It is evident that $BA$ is generally not symmetric, even when both $A$ and $B$ are symmetric positive definite. While the normalized solution can still be successfully obtained via Schr\"odingerization, the non-symmetry of $BA$ introduces additional challenges in recovering the solution.
    In particular, the matrix $(BA+(BA)^{\top})$ may possess negative eigenvalues, which can significantly increase the complexity of the problem \cite{JLM24SchrInhom}  especially for recovering our steady-state solution.}
	Similar to classical computation, it is often more advantageous to construct a symmetrically preconditioned system, particularly when the conjugate gradient (CG) method is used to solve the linear system. To address this, we propose a symmetrization technique to transform the preconditioned system into a symmetric form.

For a symmetric positive definite matrix $B$, it is well-known that there exists a unique symmetric positive definite matrix, denoted by $B^{\frac{1}{2}}$, such that $B = (B^{\frac{1}{2}})^2$. By using this result, the preconditioned system \eqref{preconditioningLSP} can be transformed into a symmetric form. Specifically, by introducing $ \tilde{\bm{x}} = B^{-\frac{1}{2}}\bm{x} $, we obtain the following system:
\begin{equation}\label{Broot}
	\begin{cases}
		\tilde{A} \tilde{\bm{x}} = \tilde{\bm{b}}, \\
		\tilde{A} = B^{\frac{1}{2}} A B^{\frac{1}{2}}, \quad \tilde{\bm{b}} = B^{\frac{1}{2}}\bm{b}.
	\end{cases}
\end{equation}
This system can be solved using the CG method in the classical setting without requiring explicit knowledge of $ B^{\frac12} $.
However, in the quantum context, it is challenging to block-encode $ B^{\frac{1}{2}} $ or derive it directly from the block encoding of $ B $. Instead, we can consider a more flexible form of $ B $, expressed as $ B = SS^{\top} $ for some matrix $ S $, which does not need to be symmetric or even square, and it is easy to verify that $$\sigma(S^{\top}A S)\backslash \{0\} = \sigma(BA).$$

\begin{theorem}\label{thm:ODESS}
	Assume that $B$ and $A$ are symmetric positive definite matrices, where $B = S S^{\top}$ with $S\in \mathbb{R}^{N\times N'}$, $N\leq N'$, and that the stationary iterative algorithm \ref{alg:SIM} converges. Define $\bm{u}(t) = S\bm{z}(t)$, where $\bm{z}(t)$ is the solution to the following linear ODE system
	\begin{equation}\label{eq:ODE square root}
		\frac{{\rm{d}}\bm{z}}{{\rm d} t} = - S^{\top}AS\bm{z} + S^{\top}\bm{b} , \qquad \bm{z}(0) = \bm{z}_0,
	\end{equation}
	 starting from $\bm{z}_0 = \bm{0}$. Then, for any $\varepsilon>0$, we have
	\begin{equation*}
		\|\bm{u}(T)-\bm{x}\| \leq \varepsilon\|\bm{x}\|,
	\end{equation*}
	where the stopping time $T$ for the evolution satisfies
	\begin{equation}\label{eq:stop T sym}
		T\geq \frac{1}{\lambda_{\min}(BA)} \log \frac{1}{\varepsilon}.
	\end{equation}
\end{theorem}
\begin{proof}
	Due to the matrix singular value decomposition,
	there exists a unitary $U_s\in \mathbb{R}^{N\times N}$ and $V_s \in \mathbb{R}^{N'
		\times N'}$ such that
	\begin{equation}\label{eq:SVD}
	    S = U_s \Sigma^{\frac{1}{2}} V_s^{\top}, \quad
	\Sigma^{\frac{1}{2}} = \begin{bmatrix}
		\Lambda^{\frac{1}{2}}&O_{N\times (N'-N)}
	\end{bmatrix},
	\end{equation}
	where $\Lambda^{\frac{1}{2}} = \text{diag}(\sqrt{\lambda_1(B)},\sqrt{\lambda_2(B)},\cdots, \sqrt{\lambda_N(B)})$.
	Then it is evident that $$B = U_s \Sigma^{\frac{1}{2}} (\Sigma^{\frac{1}{2}})^{\top} U_s^{\top}, \quad \text{and}  \quad B^{\frac{1}{2}} = U_s \Lambda^{\frac{1}{2}} U_s^{\top}.$$
	According to \eqref{Broot}, define $\tilde{\bm{u}}(t) = B^{\frac{1}{2}}\tilde{\bm{z}}(t)$, where $\tilde{\bm{z}}(t)$ satisfies the following differential equations:
		\[ \frac{{\rm{d}} \tilde{\bm{z}}}{{\rm{d}} t} = -B^{\frac{1}{2}}A B^{\frac{1}{2}} \tilde{\bm{z}} + B^{\frac{1}{2}}\bm{b}, \quad  \tilde{\bm{z}}(0)=\bm{0}.\]
	It can be readily verified that $\bm{x}=\tilde{\bm{u}}_{\infty}$ from Lemma \ref{eq:ODE}. Furthermore, the relative error satisfies the inequality:
	\[\|
	\tilde{\bm{u}}(T) - \bm{x}\| \leq \varepsilon \|\bm{x}\|,\]
    provided that the evolution time $T$ satisfies \eqref{eq:stop T sym}.
	Using the Duhamel's principle, by careful calculation, one has
	\[\tilde{\bm{u}}(t) = U_s \Lambda^{\frac{1}{2}} \int_0^t \exp \big( -\Lambda^{\frac{1}{2}} U_s^{\top} A U_s \Lambda^{\frac{1}{2}} (t-s)\big) {\rm{d}} s \; \Lambda^{\frac{1}{2}} U_s^{\top} \bm{b}.\]
	Similarly, $\bm{u}(t)=S \bm{z}(t)$ is obtained from \eqref{eq:ODE square root} by
	\begin{equation*}
		\bm{u}(t) = U_s \Sigma^{\frac{1}{2}}\int_0^t \exp \big( -(\Sigma^{\frac12})^{\top} U_s^{\top} A U_s \Sigma^{\frac{1}{2}} (t-s)\big)  {\rm{d}}s \; (\Sigma^{\frac12})^{\top} U_s^{\top}\bm{b}.
	\end{equation*}
	The proof is finished by finding that $\tilde{\bm{u}}(T) = \bm{u}(T)$.
\end{proof}

\begin{remark}
   When $S\in \mathbb{R}^{N\times N'}$  is a non-square matrix with $N<N'$, the matrix $S^\top A S$  inherently possesses eigenvalues equal to zero.
   While the linear system $ S^{\top} A S \tilde{\bm{x}} = S^{\top} \bm{b} $ may admit multiple solutions, we still have $ \bm{x} = S \tilde{\bm{x}} $ if $ S $ is of full row rank.
    From the perspective of a linear ODE system, the solution to \eqref{eq:ODE square root} is uniquely determined,
    yielding $\bm{u} = S\bm{z}$ as the unique solution.
    Moreover, Theorem \ref{thm:ODESS} ensures that the approximate solution, expressed as $\bm{u}(T) = S \bm{z}(T)$, can be effectively obtained, where the evolution time $T$ is explicitly defined in \eqref{eq:stop T sym}.
\end{remark}


According to Theorem \ref{thm:ODESS}, we can apply the Schr\"odingerization method in \cite{JLY22b,JLYPRL24} to the linear ODE system in
\eqref{eq:ODE square root} to obtain a Hamiltonian system for quantum computing.

\subsection{Schr\"odingerization method for linear ODEs}

The solution to \eqref{eq:ODE square root} can be expressed by using the Duhamel's principle as
\[\bm{z}(t) = e^{-A_S t} \bm{z}_0 + \int_0^t e^{-A_S (t-s)} \d s \bm{b}_S, \]
where $A_S = S^\top A S$ and $\bm{b}_S = S^\top \bm{b}$.  For simplicity, we choose $\bm{z}_0 = \bm{0}$ as the initial data, since the choice of the initial value does not affect the convergence.
To avoid the integration for inhomogeneous term, which requires a somewhat complicated LCU procedure, we employ the augmentation technique in \cite{JLY22b,JLYPRL24,JLM24SchrInhom} to derive a homogeneous ODE system by introducing a time-independent auxiliary vector:
\begin{equation}\label{eq:ODE2}
	\frac{{\rm d}}{{\rm d} t} \bm{z}_f = A_f \bm{z}_f,
	\quad
	A_f = \begin{bmatrix}
		-A_S  & \frac{I}{T}\\
		O & O
	\end{bmatrix},\quad
	\bm{z}_f(t) =  \begin{bmatrix}
		\bm{z}(t) \\
		T \bm{b}_S
	\end{bmatrix},
\end{equation}
where $T$ denotes the evolution time and $I$ denotes the identity matrix. By the theorem, it is sufficient to choose \eqref{eq:stop T sym}.
In our Schr\"odingerization framework, $T$ is chosen so that, for some constants $c_1,c_2>0$ independent of $\varepsilon$,
\[
c_1\,\frac{1}{\lambda_{\min}(BA)}\log\frac{1}{\varepsilon}
\;\le\;
T
\;\le\;
c_2\,\frac{1}{\lambda_{\min}(BA)}\log\frac{1}{\varepsilon}.
\]
Hence,
\begin{equation}\label{eq:schro evol Time}
    T=\Theta\!\left(\frac{1}{\lambda_{\min}(BA)}\log\frac{1}{\varepsilon}\right).
\end{equation}
It is straightforward to verify that \eqref{eq:ODE2} is equivalent to \eqref{eq:ODE}. Hence, we will focus exclusively on the homogeneous system \eqref{eq:ODE2}.

The Schr\"odingerization of the ODE system is achieved through a warped phase transformation
$\bm{v}(t,p) = \e^{-p}\bm{z}_f(t)$ for $p\ge 0$, with the equation extended naturally to $p<0$, leading to
\begin{equation}\label{eq:shcro w}
	\frac{{\rm d}}{{\rm d} t} \bm{w} =  -H_1 \partial_p \bm{w} + i H_2 \bm{w},
\end{equation}
where the Hermitian matrices $H_1$ and $H_2$  are defined by
\begin{equation}\label{eq:H1}
	H_1 = \frac{1}{2}(A_f + A_f^{\dagger}) = \begin{bmatrix}
		-A_S& \frac{I}{2T} \\
		\frac{I}{2T}  & O
	\end{bmatrix},\quad
	H_2 = \frac{1}{2i}(A_f - A_f^{\dagger}) =  \frac{1}{2i}
	\begin{bmatrix}
		O & \frac{I}{T} \\
		-\frac{I}{T} &O
	\end{bmatrix}.
\end{equation}	
The initial data for \eqref{eq:shcro w} is given by
\begin{equation}\label{initw}
\bm{w}(0,p) = \psi(p) \bm{z}_f(0),
\end{equation}
where $\psi(p)\in H^r(\mathbb{R})$, $r\geq 1$, is arbitrarily chosen, subject to the condition that $\psi(p) \approx  e^{-p}$ for $p\geq 0$. A special case for $\psi(p)\in H^1(\mathbb{R})$ is $\psi(p)=\e^{-|p|}$. For further details on the construction of $\psi$, we refer to \cite{JLM24SchrBackward,JLMY2025qSmooth}.

According to \cite[Theorem 3.1]{JLM24SchrInhom}, we can restore the solution $\bb{z}_f(t)$ by
\begin{equation}\label{eq:recovery}
	\bm{z}_f = \e^p \bm{w}(t,p),\quad p\geq p^{\Diamond}   = \lambda_{\max}^{+}(H_1)T.
\end{equation}
  Here, $\lambda_{\max}^{+}(H_1)$ denotes the largest eigenvalue of $H_1$ that is strictly positive. It can be seen from \eqref{eq:H1} that 
  \begin{equation}
  	\lambda^+_{\max}(H_1)\leq \frac{1}{2T}\quad \text{then}\quad
  	p^{\Diamond} = \frac{1}{2}.
  \end{equation}

\subsection{Discretization of the Schr\"odingerized system}
To discretize the $p$ domain, we choose a finite domain $[-R,R]$ with $R>0$ large enough satisfying \cite{JLMY2025qSmooth}
\begin{equation}\label{eq: L,R,criterion}
	e^{-R + \lambda_{\max}(H_1) T} \leq \varepsilon,
\end{equation}
where $\lambda_{\max}(H_1)$ denotes the largest absolute value among the eigenvalues of $H_1$ and $\varepsilon$ is the desired accuracy.
Then the wave initially supported inside the domain remains so in the duration of computation, and we can use spectral methods to obtain a Hamiltonian system for quantum computing.

When we truncate the extended region to a finite interval, one can apply the periodic boundary condition in the $p$ direction and use the Fourier spectral method by discretising the $p$ domain.
We set the uniform mesh size  $\triangle p = 2R/N_p$, where $N_p=2^{n_p}$ is a positive even integer. The grid points are denoted by $-R=p_0<\cdots<p_{N_p}=R$. The one-dimensional basis functions for the Fourier spectral method are usually chosen as
\begin{equation} \label{eq:phi nu}
	\phi_l(p) = e^{i\mu_l  (p+R)},\quad \mu_l  = \pi (l-N_p/2)/R,\quad 0\leq l\leq N_p-1.
\end{equation}
Using \eqref{eq:phi nu}, we define
\begin{equation}
	\Phi  = (\phi_{jl} )_{N_p\times N_p} = (\phi_l (p_j))_{N_p\times N_p},  \quad
	D_{p} = \text{diag}(\mu_0 ,\cdots,\mu_{N_p-1} ).
\end{equation}

Define the vector $\bm{W}_h$ as the collection of the approximation values of the function $\bm{w}$ at the grid points, given by
\[
	\bm{W}_h(t) =\sum_{k=0}^{N_p-1}\sum_{j=0}^{2N-1} w_{kj}(t) \ket{k}\ket{j},
\]
where $w_{kj}(t)$ denotes the approximation to $w_j(t,p_k)$, the $j$-th component of $\bm{w}(t,p_k)$.
Considering the Fourier spectral discretization on $p$, one easily gets
\begin{equation} \label{eq:hamiltonian}
	\begin{aligned}
		\frac{\rm d}{{\rm d} t} \bm{W}_h &= -i(P\otimes H_1) \bm{W}_h + i(I_{n_p} \otimes H_2)\bm{W}_h\\
		& = -i (\Phi  \otimes I)
		( D_p\otimes H_1 - I_{n_p} \otimes H_2 )
		(\Phi ^{-1}\otimes I) \bm{W}_h.
	\end{aligned}
\end{equation}
Here $P = \Phi  D_{p} \Phi ^{-1}$ is the matrix representation of the momentum operator $-i\partial_p$.
At this point, a quantum simulation can be carried out on the Hamiltonian system above:
\begin{equation*}
	\ket{\bm{W}_h(T)} = (\Phi \otimes I)~
	\mathcal{U}(T) ~
	(\Phi ^{-1}\otimes I) \ket{\bm{W}_h(0)},
\end{equation*}
where $\mathcal{U}(T)$ is a unitary operator, given by
\begin{equation}\label{eq:UT H}
	\mathcal{U}(T) = e^{-i H T}, \qquad H:= D_p\otimes H_1 - I_{n_p}\otimes H_2,
\end{equation}
and $\Phi $ (or $\Phi^{-1}$) is completed by (inverse) quantum Fourier transform (QFT or IQFT).
The complete circuit for implementing the quantum simulation of $\ket{\bm{w}_h}$ is illustrated in Fig.~\ref{schr_circuit}.
\begin{figure}[!htb]
	\centering
	\centerline{
		\Qcircuit @C=1em @R=2em {
			\lstick{\hbox to 2.7em{$\ket{\bm{\psi}_h}$\hss}}
			& \qw
			& \gate{\text{IQFT}}
			& 	\qw
			& \multigate{1}{\mathcal{U}(T)}	
			&  \qw
			& \gate{\text{QFT}}
			& \qw	
			& \qw  &\meterB{\ket{k}}\\
			\lstick{\hbox to 2.7em{$\ket{\bm{z}_f(0)}$\hss}}
			& \qw
			& \qw
			& \qw
			& \ghost{\mathcal{U}(T)}
			& \qw
			& \qw
			& \qw
			& \qw  & \hbox to 2em{$\ket{\bm{z}_f(T)}$\hss}
	}}
	\caption{Quantum circuit for Schr\"odingerization, where  $\bm{\psi}_h = \sum_{k\in [N_p]} \psi(p_k)\ket{k} $.}
	\label{schr_circuit}
\end{figure}

From \eqref{eq:recovery}, one can recover the target variables for $\bm{z}_f$ by performing a measurement in the computational basis:
\[M_{k_*} = \ket{k_*}\bra{k_*} \otimes I, \quad k_* \in \{j: p_j\geq p^{\Diamond} \;\text{and}\; p_j=\mathcal{O}(1)\}=:I_\Diamond,\]
where $I_\Diamond$ is referred to as the recovery index set.
The state vector is then collapsed to
\[ \ket{\bm{w}_*} \equiv \ket{k_*} \otimes \frac{1}{\mathcal{N}}\Big(\sum\limits_i w_{k_*i} \ket{i} \Big) , \quad
\mathcal{N} = \Big(\sum\limits_i |w_{k_* i}|^2 \Big)^{1/2},\]
where $w_{k_*,i} = \bra{k_*}\bra{i} \otimes \bb{W}_h$ approximates $e^{p_{k_*}}\bm{z}_f$
for some $k_*$ in the recovery index set $I_\Diamond$.

Next, we apply the Schr\"odingerization technique to the linear system preconditioned with the BPX preconditioner and subsequently present a quantum implementation for simulating the resulting Hamiltonian.

\section{Multilevel BPX preconditioning for finite difference method}\label{sec:BPX}

In this section, we give a brief review of the BPX preconditioner in \cite{XJ92} for a finite difference method for the Poisson equation with Dirichlet boundary condition:
\begin{equation}\label{eq:poisson eq}
	\begin{cases}
		-\Delta u = f\qquad &\text{in} \quad  \Omega,\\
		u=g_D\qquad &\text{on}\quad \partial \Omega,
	\end{cases}
\end{equation}
where $\Omega=[0,1]^d$ and the source term $f\in L^2(\Omega)$.

\subsection{Nested dyadic grids and finite difference setting}

We consider a nested sequence of uniform Cartesian grids with the step size
\[
h_\ell := 2^{-\ell},\qquad \ell=0,1,\cdots,L.
\]
The grid points are denoted by $x_{\ell,\bm i}:=h_\ell \bm i$, where $\bm i=(i_1,\cdots,i_d)\in\{0,1,\cdots,2^\ell\}^d$.
Let
\[
\mathcal{G}_\ell:=\{0,1,2,\cdots,2^{\ell}\}^d,\qquad \mathcal{I}_\ell := \{1,2,\cdots,2^\ell-1\}^d
\]
be the set of full grid indices and interior grid indices, respectively.
Accordingly, we define the multilevel spaces
\begin{equation}\label{eq:Vl}
V_\ell:=\mathbb{R}^{|\mathcal{I}_\ell|},\qquad \mathcal{I}_\ell:=\{1,2,\cdots,2^\ell-1\}^d.
\end{equation}
 The discrete unknowns on level $\ell$ are collected as the vector
$
\bm{u}_\ell := (u_{\ell,\bm i})_{\bm i\in \mathcal{I}_\ell}\in V_{\ell}.
$
For notational convenience, whenever we work on the finest level $L$ we drop the subscript $L$ and write, for instance,
\[
h:=h_L,\qquad \mathcal{G}:=\mathcal{G}_L,\qquad \mathcal{I}:=\mathcal{I}_L,\qquad x_{\bm i}:=h\,\bm i,\qquad \bm{u}:=(u_{\bm i})_{\bm i\in \mathcal{I}}\in V,
\]
and use the same convention for other level-dependent quantities when no confusion may arise.

We keep only interior degrees of freedom as unknowns and set
$
\bm{u} \in V.
$
To incorporate inhomogeneous Dirichlet data, we work with the full-grid vector
\[
\bm U := R^{\mathrm{hom}}\bm{u} + \bm g^{\mathrm{ext}} \in \mathbb{R}^{|\mathcal{G}|},
\]
where $R^{\mathrm{hom}}:\mathbb{R}^{|\mathcal{I}|}\to \mathbb{R}^{|\mathcal{G}|}$ denotes the homogeneous embedding and
$\bm g^{\mathrm{ext}}\in\mathbb{R}^{|\mathcal{G}|}$ is the discrete extension of the boundary values. They are defined by
\[
(R^{\mathrm{hom}}\bm{u})_{\bm i}:=
\begin{cases}
	u_{\bm i}, & \bm i\in \mathcal{I},\\
	0, & \bm i\in \mathcal{G}\setminus \mathcal{I},
\end{cases}
\qquad
g^{\mathrm{ext}}_{\bm i}:=
\begin{cases}
	g_D(x_{\bm i}), & \bm i\in \mathcal{G}\setminus \mathcal{I},\\
	0, & \bm i\in \mathcal{I}.
\end{cases}
\]
With this lifting, the stiffness matrix remains exactly the same as in the homogeneous Dirichlet case; the
inhomogeneous boundary data enter only through the right-hand side.

\subsection{Finite difference discretization and stiffness matrix.}

For each interior index $\bm i\in \mathcal{I}$, we use the central difference along each direction to discretize the Laplacian. For ease of constructing the multilevel preconditioner, we include an energy scaling as for the finite element discretization in the discrete Laplacian:
\begin{equation}\label{disLap}
(\Delta_h \bm U)_{\bm i}
:= h^{d-2}\sum_{k=1}^d\bigl(U_{\bm i+\bm e_k}-2U_{\bm i}+U_{\bm i-\bm e_k}\bigr),
\end{equation}
where $\bm e_k$ is the $k$-th canonical unit vector in $\mathbb{R}^d$. This yields the linear system
\begin{equation}\label{eq:Au=b}
	A\,\bm{u}=\bm{b},
\end{equation}
where $A\in\mathbb{R}^{|\mathcal{I}|\times|\mathcal{I}|}$ and $\bm b\in\mathbb{R}^{|\mathcal{I}|}$ are given by
\[
(A\bm{u})_{\bm i} := -\bigl(\Delta_h(R^{\mathrm{hom}}\bm{u})\bigr)_{\bm i},
\qquad
\bm b_{\bm i} := h^d f(x_{\bm i}) + \bigl(\Delta_h \bm g^{\mathrm{ext}}\bigr)_{\bm i},
\qquad \bm i\in \mathcal{I}.
\]
In particular,
\[
A_{\bm i,\bm i}=2d\,h^{d-2},\qquad
A_{\bm i,\bm i\pm \bm e_k}=-h^{d-2}\quad (k=1,\cdots,d),
\]
whenever $\bm i\pm\bm e_k\in \mathcal{I}$, and boundary contributions are incorporated into $\bm b$ through $\bm g^{\mathrm{ext}}$.

On the uniform Cartesian grid, the same stiffness matrix $A$ admits an explicit tensor-product form.
Let $A^{1d}\in\mathbb{R}^{(2^L-1)\times(2^L-1)}$ denote the one-dimensional homogeneous-Dirichlet stiffness matrix
in the same scaling,
\[
(A^{1d}\bm{w})_i := h^{d-2}\bigl(2w_i-w_{i-1}-w_{i+1}\bigr),\qquad i=1,\cdots,2^L-1,
\]
with the convention $w_0=w_{2^L}=0$. After fixing a vectorization
$\mathbb{R}^{|\mathcal{I}|}\cong \mathbb{R}^{2^L-1}\otimes\cdots\otimes\mathbb{R}^{2^L-1}$, the $d$-dimensional operator can be
written as the Kronecker sum:
\begin{equation}\label{eq:Kronecker_sum_A}
	A  =  \sum_{k=1}^d I^{\otimes (k-1)}\otimes A^{1d}\otimes I^{\otimes (d-k)},
\end{equation}
where $I$ is the identity matrix of size $(2^L-1)\times(2^L-1)$.

\begin{remark}\label{rem:tensor_reduction}
	The Kronecker-sum form shows that the $d$-dimensional operator is built from identical one-dimensional components.
	Accordingly, it suffices to derive the multilevel ingredients in one dimension and lift them to $d$ dimensions via
	standard Kronecker-product rules. This tensor structure is also well suited for quantum implementations: tensor-product
	block-encodings act in parallel on separate registers, and the Kronecker sum is handled by a simple selector over the
	$d$ tensor factors.
\end{remark}


\subsection{BPX preconditioner for finite differences}
Recall that, under the lifting formulation, the stiffness matrix $A$ is exactly the homogeneous-Dirichlet finite difference
stiffness matrix on interior nodes. Therefore, on the multilevel spaces $V_\ell$ defined in \eqref{eq:Vl}, we employ
prolongation operators given by standard nodal interpolation on nested dyadic grids, using the homogeneous endpoint
convention at the boundary.

Let $P_\ell:V_\ell\to V_L$ be the nodal prolongation from level $\ell$ to the finest level $L$. Considering the fact that the energy scaling has been included in \eqref{disLap}, we can define the BPX preconditioner as for the FEM \cite{XJ92}:
\begin{equation}\label{eq:BPX}
	B \;:=\; \sum_{\ell=0}^L h_\ell^{2-d}\,P_\ell P_\ell^\top .
\end{equation}
For the symmetric preconditioning used later, it is convenient to write \eqref{eq:BPX} in factored form
\begin{equation}\label{eq:B_factor}
	B = SS^\top,\qquad
	S=\bigl[S_0,\;S_1,\;\cdots,\;S_L\bigr],\qquad
	S_\ell:=h_\ell^{(2-d)/2}P_\ell .
\end{equation}

On nested uniform dyadic grids, $P_\ell$ is the multilinear interpolation operator and admits the tensor-product form
\begin{equation}\label{eq:P_tensor}
	P_\ell \;=\; \underbrace{P_\ell^{1d}\otimes \cdots \otimes P_\ell^{1d}}_{d\ \text{times}},
\end{equation}
where $P_\ell^{1d}:\mathbb{R}^{2^\ell-1}\to \mathbb{R}^{2^L-1}$ is the one-dimensional prolongation and can be written as a product of successive one-step prolongations.
For $\ell=0,1,\cdots,L-1$, define the one-step prolongation
\[
P_{\ell\to \ell+1}^{1d}\in\mathbb{R}^{(2^{\ell+1}-1)\times(2^\ell-1)}
\]
by the standard linear interpolation with homogeneous boundary values at the endpoints.
Equivalently, for $\bm v\in\mathbb{R}^{2^\ell-1}$, the fine-grid vector $\bm{w}=P_{\ell\to \ell+1}^{1d}\bm v\in\mathbb{R}^{2^{\ell+1}-1}$
satisfies
\begin{equation}\label{eq:p_l2next}
w_{2q}=v_q,\qquad q=1,\cdots,2^\ell-1,
\qquad
w_{2q-1}=\tfrac12\,(v_{q-1}+v_q),\qquad q=1,\cdots,2^\ell,
\end{equation}
where $v_0=v_{2^\ell}=0$.
Thus, $P_{\ell\to \ell+1}^{1d}$ is the tridiagonal-type matrix
\begin{equation}\label{eq:P_one_step_1d}
	P_{\ell\to \ell+1}^{1d}
	=
	\begin{pmatrix}
		\frac12 & 0 & 0 & \cdots & 0\\
		1 & 0 & 0 & \cdots & 0\\
		\frac12 & \frac12 & 0 & \cdots & 0\\
		\vdots & \ddots & \ddots & \ddots & 0\\
		\vdots & & 0 & \frac12 & \frac12\\
		0 & \cdots & 0 & 0 & 1\\
		0 & \cdots & 0 & 0 & \frac12
	\end{pmatrix}
	\in\mathbb{R}^{(2^{\ell+1}-1)\times(2^\ell-1)} .
\end{equation}
Finally, the coarse-to-fine prolongation is the product of the one-step operators,
\begin{equation}\label{eq:P1d_product}
	P_\ell^{1d} \;=\; P_{L-1\to L}^{1d}\,P_{L-2\to L-1}^{1d}\cdots P_{\ell\to \ell+1}^{1d}.
\end{equation}
Combined with \eqref{eq:P_tensor}, this gives an explicit tensor-product construction of $P_\ell$ in $d$ dimensions and
hence of the BPX preconditioner \eqref{eq:BPX}.

The classical BPX theory \cite{XJ92} implies that, under the above uniform refinement and standard
 interpolation operators, the preconditioned system has a uniformly bounded condition number, i.e.,
 \[
 \kappa(BA)\le C,
 \]
 where $C>0$ is independent of the finest level $L$ and may depend polynomially on the spatial dimension $d$.
Recent work shows that, in the finite-element setting, the dimension dependence of BPX preconditioners is polynomial in \(d\) \cite{JiangParkXu2025}. This suggests that, in the present finite-difference setting, the spectrum of \(BA\) should remain bounded away from \(0\) and \(\infty\) up to polynomial factors in \(d\), namely
\begin{equation}\label{eq:eig BA}
\lambda_{\min}(BA)\gtrsim \mathrm{poly}(d)^{-1},
\qquad
\lambda_{\max}(BA)\lesssim \mathrm{poly}(d),
\end{equation}
and hence
\[
\kappa(BA)\lesssim \mathrm{poly}(d).
\]

\section{Block encoding of the preconditioned system}\label{sec:BE}
Our quantum constructions are formulated in terms of block-encodings. Given a block-encoding of a matrix $H$, standard
techniques can be used to implement $\mathcal{U}(T)=e^{-iHT}$ and related primitives; see, for example,
\cite{BCK15,gilyen2019quantum}. In the preconditioning setting studied here, we will use structure-aware encodings.
Sparse-access block-encodings remain an essential tool and will be used whenever the relevant matrix is sparse.

\subsection{Block-encoding preliminaries}\label{subsec:blockencoding}

We first recall the standard notion of block-encoding for square matrices and the two basic tools that we will use
throughout, namely sparse-access constructions and matrix arithmetic rules.

\begin{definition}[Block-encoding]\label{def:blockencoding}
	Let $A\in\mathbb{C}^{N\times N}$ and let $\Pi:=\bra{0^m}\otimes I_n$, where $N \le 2^n$ and $I_n$ is an $n$-qubit identity matrix. A unitary $U_A$ on $m+n$ qubits is an
	$(\alpha,m,\varepsilon)$-block-encoding of $A$ if
	\[
	\bigl\|A-\alpha\,\Pi U_A \Pi^\dagger\bigr\|\le \varepsilon,
	\qquad \alpha\ge \|A\|.
	\]
\end{definition}

Most operators in our construction are rectangular. We therefore use \emph{generalized block-encodings}
(a.k.a.\ projected unitary encodings) \cite{gilyen2019quantum,DP24}, which directly encode rectangular maps by allowing
distinct input and output subspaces.

\begin{definition}[Generalized block-encoding]\label{def:gen_blockencoding}
	Let $S\in\mathbb{C}^{M_2\times M_1}$ with $M_1,M_2\le M:=2^m$. Let $\Pi_1:\mathbb{C}^M\to\mathbb{C}^{M_1}$ and
	$\Pi_2:\mathbb{C}^M\to\mathbb{C}^{M_2}$ be fixed orthogonal projectors. A unitary $U_S$ on $m$ qubits is called a
	generalized $(\gamma,\varepsilon)$-block-encoding of $S$ with respect to $(\Pi_1,\Pi_2)$ if there exists $\gamma>0$ such that
	\[
	\bigl\|S-\gamma\,\Pi_2 U_S \Pi_1^\dagger\bigr\|\le \varepsilon.
	\]
	If $\varepsilon=0$, we call $(U_S,\Pi_1,\Pi_2)$ an exact generalized block-encoding of $S$.
\end{definition}

\begin{remark} 
The above definition assumes the availability of a quantum random access memory (QRAM) that allows coherent queries to classical data in superposition. QRAM is a well-established theoretical model in quantum computing and is widely used in quantum algorithms for linear algebra and differential equations. However, it is important to note that a practical, fault-tolerant realization of QRAM remains an open experimental challenge on current quantum devices. The results in this work should therefore be understood as theoretical contributions within the standard QRAM-based quantum algorithm framework.
\end{remark}

We recall the sparse-query model and the resulting block-encoding construction \cite{gilyen2019quantum,CGJ19}.
\begin{lemma}[Sparse-access generalized block-encoding for rectangular matrices]\label{lem:sparse_access_be_rect}
	Let $A\in\mathbb{C}^{2^{w_2}\times 2^{w_1}}$ be $s_r$-row-sparse and $s_c$-column-sparse, and assume $|A_{ij}|\le 1$.
	Assume sparse-access oracles
	\[
	O_r:\ |i\rangle|k\rangle \mapsto |i\rangle|r_{ik}\rangle,\quad i\in[2^{w_2}],~k\in[s_r],
	\qquad
	O_c:\ |\ell\rangle|j\rangle \mapsto |c_{\ell j}\rangle|j\rangle,\quad \ell\in[s_c],~j\in[2^{w_1}],
	\]
	where $r_{ik}\in[2^{w_1}]$ indexes the $k$-th nonzero column position in row $i$,
	and $c_{\ell j}\in[2^{w_2}]$ indexes the $\ell$-th nonzero row position in column $j$,
	with standard dummy indices if fewer nonzeros exist (as in \cite{gilyen2019quantum}). Additionally, assume access to an entry oracle
	\[
	O_A:\ |i\rangle|j\rangle|0\rangle^{\otimes b} \mapsto |i\rangle|j\rangle|a_{ij}\rangle,\quad i\in[2^{w_2}],~j\in[2^{w_1}].
	\]
	Then one can implement a generalized $(\sqrt{s_rs_c},\varepsilon)$-block-encoding of $A$
	(with respect to the natural input/output projectors onto $\mathbb{C}^{2^{w_1}}$ and $\mathbb{C}^{2^{w_2}}$)
	with a single use of $O_r$, a single use of $O_c$, two uses of $O_A$, and additionally using
	$\mathcal{O}\!\left(2+\log^{2.5}\!\bigl(\tfrac{s_r s_c}{\varepsilon}\bigr)\right)$ one- and two-qubit gates, while using
	$\mathcal{O}\!\left(b,\log^{2.5}\!\bigl(\tfrac{s_r s_c}{\varepsilon}\bigr)\right)$ ancilla qubits.
\end{lemma}

We will repeatedly combine generalized block-encodings using the following rules.

\begin{lemma}[Matrix arithmetic for generalized block-encodings]\label{lem:gen_arithmeticBE}
	Let $U_1$ be a generalized $(\gamma_1,\varepsilon_1)$-block-encoding of $S_1$ with respect to $(\Pi_0,\Pi_1)$,
	and $U_2$ be a generalized $(\gamma_2,\varepsilon_2)$-block-encoding of $S_2$ with respect to $(\Pi_1,\Pi_2)$.
	Then:
	\begin{enumerate}
		\item \textbf{Adjoint.}
		$U_1^\dagger$ is a generalized $(\gamma_1,\varepsilon_1)$-block-encoding of $S_1^\dagger$ with respect to $(\Pi_1,\Pi_0)$.
		
		\item \textbf{Product (composition).}
		$U_2U_1$ is a generalized
		$
		\bigl(\gamma_1\gamma_2,\ \gamma_2\varepsilon_1+\gamma_1\varepsilon_2\bigr)$-block-encoding of  $S_2S_1$
		with respect to $(\Pi_0,\Pi_2)$.
		
		\item \textbf{Tensor product.}
		Let $U_1'$ be a generalized $(\gamma_1',\varepsilon_1')$-block-encoding of $S_1'$ with respect to $(\Pi_0',\Pi_1')$,
		acting on a disjoint register. Then $U_1\otimes U_1'$ (acting on disjoint registers) is a generalized
		$
		\bigl(\gamma_1\gamma_1',\ \gamma_1'\varepsilon_1+\gamma_1\varepsilon_1'\bigr)$-block-encoding of  $S_1\otimes S_1'$
		with respect to $(\Pi_0\otimes \Pi_0',\ \Pi_1\otimes \Pi_1')$.
		
		\item \textbf{Horizontal concatenation $[A~B]$.}
       Assume that $U_A$ and $U_B$ are exact generalized block-encodings of
$A\in\mathbb{C}^{M\times N_1}$ and $B\in\mathbb{C}^{M\times N_2}$, respectively, with normalizations
$\gamma_A$ and $\gamma_B$, and with respect to the projector pairs
$(\Pi_{\mathrm{in},A},\Pi_{\mathrm{out}})$ and $(\Pi_{\mathrm{in},B},\Pi_{\mathrm{out}})$.
Introduce a selector qubit $\mathsf{b}$ and define the direct-sum input projector
$\Pi_{\mathrm{in},A}\oplus \Pi_{\mathrm{in},B}:\mathbb{C}^{2^m}\to\mathbb{C}^{N_1+N_2}$ by
$
(\Pi_{\mathrm{in},A}\oplus \Pi_{\mathrm{in},B})^\dagger
=
\ket{0}\!\bra{0}_{\mathsf{b}}\otimes \Pi_{\mathrm{in},A}^\dagger
\;+\;
\ket{1}\!\bra{1}_{\mathsf{b}}\otimes \Pi_{\mathrm{in},B}^\dagger.
$
Then the horizontally concatenated matrix $[A~B]\in\mathbb{C}^{M\times (N_1+N_2)}$ admits an exact generalized block-encoding
with normalization
$
\gamma_{[A~B]}=\sqrt{\gamma_A^2+\gamma_B^2},
$
with respect to $(\Pi_{\mathrm{in},A}\oplus \Pi_{\mathrm{in},B},\,\Pi_{\mathrm{out}})$.
	\end{enumerate}
\end{lemma}

\begin{remark}\label{remark:Vertical stacking}
	Note that
	\[
	\begin{bmatrix}A\\ B\end{bmatrix}=\bigl[\,A^\dagger\ \ B^\dagger\,\bigr]^\dagger .
	\]
	If $U_A$ and $U_B$ are exact generalized block-encodings of $A$ and $B$ sharing the same input projector
	$\Pi_{\mathrm{in}}$, then by Lemma~\ref{lem:gen_arithmeticBE}(1) the unitaries $U_A^\dagger$ and $U_B^\dagger$
	are exact generalized block-encodings of $A^\dagger$ and $B^\dagger$ sharing the same output projector $\Pi_{\mathrm{in}}$.
	Applying Lemma~\ref{lem:gen_arithmeticBE}(4) to $[A^\dagger\ B^\dagger]$ and taking the adjoint again yields an exact
	generalized block-encoding of $\bigl[\begin{smallmatrix}A\\B\end{smallmatrix}\bigr]$ with normalization
	\[
	\gamma_{\bigl[\begin{smallmatrix}A\\B\end{smallmatrix}\bigr]}=\sqrt{\gamma_A^2+\gamma_B^2}.
	\]
\end{remark}

{ \subsection{ Block-encoding of the preconditioned matrix $S^{\top} A S$}
This section focuses  on efficiently constructing the block-encoding of $S^{\top} A S$. This directly relates to  the Hamiltonian simulation. A naive approach would block-encode $S$ and $A$ separately and
multiply them to obtain $A_S$, but this incurs a normalization scaling like $\|S\|^2\|A\|$, which is essentially
$\kappa(A)$ because $S\approx A^{-1/2}$ \cite{DP24}. To avoid this $\kappa(A)$-dependence, we follow the idea in \cite{DP24} to use the factorization $A=C_L^\top C_L$ and construct the block-encoding at the level of the composite operator $C_L S$, so that
\[
A_S=S^\top A S=(C_L S)^\top(C_L S).
\]
This will yield a block-encoding of $A_S$ with normalization $\mathcal{O}(\kappa(BA))=\mathcal{O}(1)$, and therefore an
efficient block-encoding of $H_1$.

Next, we will discuss the specific construction of $ C_L $ and how to encode the matrix product $ C_L S $.

\subsubsection{Factorization of the stiffness matrix $A$}
In this subsection, we prove the factorization
\begin{equation}\label{eq:A_CLtCL}
	A = C_L^\top C_L,
\end{equation}
which will be crucial for constructing efficient block-encodings. The proof relies on the tensor-product structure in \eqref{eq:Kronecker_sum_A}. We begin with an explicit one-dimensional construction, then extend it to $ d $-dimensions using Kronecker products.

On the finest level $ L $ in one dimension, the interior space is $ V_L^{1d} = \mathbb{R}^{2^L-1} $, and we write $ \bm{u} = (u_1, \cdots, u_{2^L-1})^\top \in V_L^{1d} $. The homogeneous embedding introduced earlier fixes $ u_0 = u_{2^L} = 0 $ at the endpoints. Define the forward difference operator composed with the homogeneous embedding by
 \begin{equation}\label{eq:def DL}
 (D_L^{1d} \bm{u})_j = (R^{\mathrm{hom}} \bm{u})_{j+1} - (R^{\mathrm{hom}} \bm{u})_{j},\quad   j = 0, \cdots, 2^L-1.
 \end{equation}
With the energy scaling used in \eqref{eq:Au=b}, we set:
\begin{equation}\label{eq:CL_1d}
	C_L^{1d} := h^{\frac{d}{2} - 1} D_L^{1d}.
\end{equation}
Direct computation gives $	A^{1d} = (C_L^{1d})^\top C_L^{1d} = h^{d-2} (D_L^{1d})^\top D_L^{1d}$, with
\begin{equation*}\label{eq:A1d_explicit}
D_L^{1d}=\hspace{-1mm}		\begin{pmatrix}
		1 &  &  &  \\
		-1 & 1 &  &   \\
		& \ddots & \ddots &   \\
		&  &  -1& 1 \\
		&  &  \hspace{-1mm}&-1
	\end{pmatrix}\hspace{-1mm}\in \mathbb{R}^{2^L \times (2^L - 1)},\;
A^{1d}	= h^{d-2}\hspace{-1mm}
\begin{pmatrix}
		2 & -1 \\
		-1 & 2 & -1 \\
		& \ddots & \ddots & \ddots \\
		&& -1 & 2 & -1 \\
		&&& -1 & 2
	\end{pmatrix}\hspace{-1mm}
	\in \mathbb{R}^{(2^L-1) \times (2^L-1)}.
\end{equation*}

For the $ d $-dimensional case, the stiffness matrix admits the Kronecker-sum form \eqref{eq:Kronecker_sum_A}. The directional difference operators for each direction $ k = 1, \cdots, d $ are defined as:
\begin{equation}\label{eq:DL}
	D_L^{(k)} := I^{\otimes (k-1)} \otimes D_L^{1d} \otimes I^{\otimes (d-k)},
	\qquad
	C_L^{(k)} := h^{\frac{d}{2} - 1} D_L^{(k)}.
\end{equation}
The overall operator $ C_L $ is then formed by stacking the directional operators:
\begin{equation}\label{eq:DL_stack}
	D_L :=
	\begin{bmatrix}
		(D_L^{(1)})^{\top}, \cdots
		,(D_L^{(d)})^{\top}
	\end{bmatrix}^{\top},
	\quad
	C_L := h^{\frac{d}{2} - 1} D_L =
	\begin{bmatrix}
		(C_L^{(1)})^{\top}, \cdots
		,(C_L^{(d)})^{\top}
	\end{bmatrix}^\top.
\end{equation}
Using \eqref{eq:A1d_explicit}, the Kronecker-product property, and \eqref{eq:Kronecker_sum_A}, we obtain
\begin{align*}
C_L^\top C_L
& = \sum_{k=1}^d (C_L^{(k)})^\top C_L^{(k)} = \sum_{k=1}^d I^{\otimes (k-1)} \otimes (C_L^{1d})^\top C_L^{1d} \otimes I^{\otimes (d-k)} \\
& = \sum_{k=1}^d I^{\otimes (k-1)} \otimes A^{1d} \otimes I^{\otimes (d-k)} = A,
\end{align*}
which proves \eqref{eq:A_CLtCL}. Similarly, we define $ D_{\ell}^{(k)} $ and $ D_{\ell} $ on the $ \ell $-th level grid in an analogous manner.

\subsubsection{Commuting identity}
To block-encode
\[
C_L S=[C_L S_0, C_LS_1, \cdots, C_LS_L],
\]
we only need to consider each level component. Using $C_L = h_L^{(d-2)/2}D_L$ and the definition of the scaled prolongations in $S$, we have
\begin{equation}\label{eq:CLS_expand}
	C_L S_\ell
	= r^{(2-d)/2}\,D_L P_\ell,
	\qquad r:=2^{L-\ell}.
\end{equation}
A direct product of block-encodings $C_L$ and $S_\ell$, i.e., $D_L$ and $P_\ell$, leads to an unfavorable normalization.
Following the similar idea in \cite{DP24}, we exploit a commuting identity that rewrites $D_LP_\ell$
in terms of the coarse-level derivative $D_\ell$ and the matrix $E_{\ell \to L}$, which accounts for the change of level
from $L$ to $\ell$.

\begin{lemma}\label{lem: One-Dimensional Identity}
	Let $ D_L^{1d} $ and $ D_\ell^{1d} $ be the discrete derivative operators at the fine and coarse levels, respectively. Let $ P_\ell^{1d} $ be the prolongation matrix that interpolates between the coarse and fine levels. Then the following core identity holds:
	\begin{equation}\label{eq:1d identity}
		D_L^{1d} P_\ell^{1d} = \frac{1}{r} E_{\ell \to L}^{(1d)} D_\ell^{1d},
	\end{equation}
	where $ r = 2^{L - \ell} $, and $ E_{\ell \to L}^{(1d)} $ is the interpolation matrix that copies coarse grid values to the fine grid, written as
	\[
	E_{\ell \to L}^{(1d)} = I_{\ell} \otimes \mathbf{1}_{r}.
	\]
	where $I_{\ell}$ is the identity matrix of size $2^\ell$, and $\mathbf{1}_r$ is a column vector of length $r$ containing all ones.
\end{lemma}

\begin{proof}
The identity \eqref{eq:1d identity} is due to the linearity of the prolongation.
We verify it by evaluating the actions on the grid functions.
	
	For a given coarse grid vector $ \bm{u}_\ell \in V_\ell^{1d} $, let $ \widetilde{\bm{u}}_\ell = R^{\mathrm{hom}} \bm{u}_\ell $ be its homogeneous extension with boundary values set to zero. The discrete derivative at the coarse level is:
	\[
	(D_\ell^{1d} \bm{u}_\ell)_q = \widetilde{u}_{\ell,q+1} - \widetilde{u}_{\ell,q}, \quad q = 0, \cdots, 2^\ell - 1.
	\]
	The matrix $ E_{\ell \to L}^{(1d)} = I_{2^\ell} \otimes \mathbf{1}_{r} $ maps each coarse grid edge to $ r $ consecutive fine grid edges.
    Set $r=2^{L-\ell}$. Each coarse interval edge at level $\ell$ is uniformly subdivided into $r$ fine intervals at level $L$.
    We index coarse intervals by $q\in\{0,\cdots,2^\ell-1\}$ and fine intervals by $j'\in\{0,\cdots,2^L-1\}$.
    Every fine-edge index $j'$ can be written uniquely in the quotient--remainder form
    \[
     j' = q r + t,\qquad q=\bigl\lfloor j'/r\bigr\rfloor,\qquad t=j'\bmod r,\qquad t\in\{0,\cdots,r-1\}.
     \]
    Thus $q$ identifies the parent coarse interval, while $t$ is the local index of the fine interval within that coarse interval.
    For each fine edge $ j' = q r + t $ with $ t = 0, \cdots, r-1 $, we have
	\[
	(E_{\ell \to L}^{(1d)} D_\ell^{1d} \bm{u}_\ell)_{j'} = (D_\ell^{1d} \bm{u}_\ell)_q.
	\]
	
	Next, consider the prolongation matrix $ P_\ell^{1d} $, which interpolates between the coarse and fine grids. The value at a fine grid node $ j' = q r + t $ within the coarse cell $ q $ is:
	\begin{equation}\label{eq: P_1 1d}
	(P_\ell^{1d} \bm{u}_\ell)_{j'} = \frac{r-t}{r} \widetilde{u}_{\ell,q} + \frac{t}{r} \widetilde{u}_{\ell,q+1}.
	\end{equation}
	For the fine grid edges, the discrete derivative is:
	\[
	(D_L^{1d} P_\ell^{1d} \bm{u}_\ell)_{j'} = (P_\ell^{1d} \bm{u}_\ell)_{j'+1} - (P_\ell^{1d} \bm{u}_\ell)_{j'}.
	\]
	For $ j' = q r + t $ with $ 0 \le t \le r-1 $, we compute:
	\[
	(P_\ell^{1d} \bm{u}_\ell)_{j'+1} - (P_\ell^{1d} \bm{u}_\ell)_{j'} = \frac{1}{r} (\widetilde{u}_{\ell,q+1} - \widetilde{u}_{\ell,q}) = \frac{1}{r} (D_\ell^{1d} \bm{u}_\ell)_q.
	\]
	This identity holds for the last edge of a coarse cell ($ t = r-1 $), where $ j'+1 $ corresponds to a coarse node, and $ (P_\ell^{1d} \bm{u}_\ell)_{j'+1} = \widetilde{u}_{\ell,q+1} $.
	
	Therefore, for every fine edge $ j' $, we have
	\[
	(D_L^{1d} P_\ell^{1d} \bm{u}_\ell)_{j'} = \frac{1}{r} (E_{\ell \to L}^{(1d)} D_\ell^{1d} \bm{u}_\ell)_{j'}.
	\]
	This completes the proof.
\end{proof}

%

For the $ d $-dimensional case,
the image-space inter-level operator for direction $ k $, denoted by $ E_{\ell \to L}^{(k)} $, must handle interpolation weights in the other $ d-1 $ directions, whose precise form is
\begin{equation}\label{eq:E-tensor}
E_{\ell \to L}^{(k)} = \Big( \bigotimes_{m=1}^{k-1} P_\ell^{1d} \Big) \otimes E_{\ell \to L}^{(1d)} \otimes \Big( \bigotimes_{m=k+1}^{d} P_\ell^{1d} \Big).
\end{equation}
Using the mixed-product property and Lemma \ref{lem: One-Dimensional Identity}, we obtain
\begin{equation}\label{eq:core-d-tensor}
	\begin{aligned}
	D_L^{(k)} P_\ell &= \Bigl(I^{\otimes (k-1)} \otimes D_L^{1d} \otimes I^{\otimes (d-k)}\Bigr) \Bigl(\bigotimes\limits_{m=1}^d P_\ell^{1d}\Bigr) \\
 &	= \frac{1}{r}
	\Bigl( \bigotimes_{m=1}^{k-1} P_\ell^{1d} \Bigr)  \otimes \Bigl( E_{\ell \to L}^{(1d)}D_{\ell}^{1d}\Bigr)  \otimes \Bigl( \bigotimes_{m=k+1}^{d} P_\ell^{1d}\Bigr) \\
	&
	= \frac{1}{r} E_{\ell \to L}^{(k)} D_\ell^{(k)},
	\end{aligned}
\end{equation}
where $ D_{\ell}^{(k)}=I^{\otimes (k-1)} \otimes D_{\ell}^{1d} \otimes I^{\otimes (d-k)} $, $ \ell = 0, 1, \cdots, L $.

The overall operator $ E_{\ell \to L} $ is a block diagonal matrix formed by the directional operators. Specifically, 
\begin{equation}\label{eq:E_L}
	E_{\ell \to L} = \text{diag}\left(E_{\ell \to L}^{(1)}, \cdots, E_{\ell \to L}^{(d)}\right).
\end{equation}
With the above tensor product definitions, 
\begin{equation}\label{eq:commuting_stacked}
	D_LP_\ell=\frac{1}{r}\,E_{\ell \to L}\,D_\ell.
\end{equation}

\begin{theorem}[Commuting identity for $C_LS_\ell$]\label{thm:CS=TD}
	Let $r:=2^{L-\ell}$ and let $E_{\ell \to L}$ be defined in \eqref{eq:E_L}. Then
	\begin{equation}\label{eq:CS_TD}
		C_LS_\ell = T_{\ell\to L}\,D_\ell,
		\qquad
		T_{\ell\to L}:=r^{-d/2}\,E_{\ell \to L}.
	\end{equation}
\end{theorem}

\begin{proof}
	By \eqref{eq:CLS_expand} and \eqref{eq:commuting_stacked},
	\[
	C_LS_\ell
	=r^{(2-d)/2}\,D_LP_\ell
	=r^{(2-d)/2}\cdot\frac{1}{r}\,E_{\ell \to L}D_\ell
	=r^{-d/2}\,E_{\ell \to L}D_\ell
	=T_{\ell\to L}D_\ell,
	\]
as required.
\end{proof}

\subsubsection{Implementation of the preconditioned matrix}

Based on the above discussion, we use Theorem \ref{thm:CS=TD} to implement the block-encoding of $C_L S_\ell$.
and then $S^{\top} A S = (C_L S)^{\top} (C_L S)$.

\begin{lemma}\label{lem:BE_P_one_step_sqrt2}
Let $\ell\ge 0$ and set
\[
m_\ell:=2^\ell-1,\qquad m_{\ell+1}:=2^{\ell+1}-1.
\]
Let $P_{\ell\to \ell+1}^{1d}\in\mathbb{R}^{m_{\ell+1}\times m_\ell}$ be the one-dimensional nodal linear interpolation
(prolongation) on interior points with homogeneous extension at the two endpoints, as in
\eqref{eq:p_l2next}--\eqref{eq:P_one_step_1d}. Then $\|P_{\ell\to \ell+1}^{1d}\|\le \sqrt{2}$.
Moreover, there exists an integer $m=\mathcal{O}(\ell)$, an $m$-qubit unitary $U_P$, and orthogonal projectors
$\Pi_{\mathrm{in}}:\mathbb{C}^{2^m}\to \mathbb{C}^{m_\ell}$ and $\Pi_{\mathrm{out}}:\mathbb{C}^{2^m}\to \mathbb{C}^{m_{\ell+1}}$
such that
\[
P_{\ell\to \ell+1}^{1d}=\sqrt{2}\,\Pi_{\mathrm{out}}\,U_P\,\Pi_{\mathrm{in}}^\dagger.
\]
Equivalently, $(U_P,\Pi_{\mathrm{in}},\Pi_{\mathrm{out}})$ is an exact generalized $(\sqrt{2},0)$-block-encoding of
$P_{\ell\to \ell+1}^{1d}$ in the sense of Definition~\ref{def:gen_blockencoding}.
\end{lemma}

\begin{proof}
Write $P:=P_{\ell\to \ell+1}^{1d}$.
A direct computation shows that $P^\top P$ is tridiagonal with diagonal entries $3/2$ and off-diagonal entries $1/4$.
Hence every row sum of absolute values is at most $2$, so $\|P^\top P\|\le 2$ and $\|P\|\le \sqrt{2}$.

Set $N=2^{\ell+1}$. Let $\mathsf{S}$ be a $w$-qubit register with computational basis $\{\ket{j}\}_{j=0}^{N-1}$.
We identify $\mathbb{R}^{m_\ell}$ and $\mathbb{R}^{m_{\ell+1}}$ with the subspaces
\[
\mathcal{H}_{\mathrm{in}}:=\mathrm{span}\{\ket{q}:1\le q\le m_\ell\}\subset \mathbb{C}^{N},
\qquad
\mathcal{H}_{\mathrm{out}}:=\mathrm{span}\{\ket{j}:1\le j\le m_{\ell+1}\}\subset \mathbb{C}^{N}.
\]
Let $\Pi_{\mathrm{in}}$ and $\Pi_{\mathrm{out}}$ be the corresponding projectors
\[
\Pi_{\mathrm{in}}=\sum_{q=1}^{m_\ell}\ket{q}\!\bra{q}\in \mathbb{R}^{m_l\times N},
\qquad
\Pi_{\mathrm{out}}=\sum_{j=1}^{m_{\ell+1}}\ket{j}\!\bra{j}\in \mathbb{R}^{m_{l+1}\times N}.
\]

Define the permutation $\pi$ on $\{1,\dots,m_{\ell+1}\}$ by listing all even indices first and all odd indices second:
\[
\pi(2q)=q,\qquad q=1,\dots,m_\ell,
\qquad
\pi(2q-1)=m_\ell+q,\qquad q=1,\dots,2^\ell.
\]
Let $U_\pi\in\mathbb{C}^{N\times N}$ be the corresponding permutation unitary defined by $U_\pi\ket{j}=\ket{\pi(j)}$ for
$1\le j\le m_{\ell+1}$ and $U_\pi\ket{j}=\ket{j}$ for $j=0$ and $j>m_{\ell+1}$.
Let $P_e\in\mathbb{R}^{N\times N}$ be the standard embedding of $P$, namely
\[
(P_e)_{j,q}=P_{j,q}\quad (1\le j\le m_{\ell+1},\ 1\le q\le m_\ell),\qquad (P_e)_{j,q}=0\ \text{otherwise}.
\]
Define
\[
P_{\mathrm{reorder}}:=\Pi_{\mathrm{out}}\,U_\pi\,P_e\,\Pi_{\mathrm{in}}^\dagger\in\mathbb{R}^{m_{\ell+1}\times m_\ell}.
\]
Then, by the definition of $P$,
\[
P_{\mathrm{reorder}}
=
\begin{bmatrix}
I_{m_\ell\times m_\ell}\\
Q_0
\end{bmatrix},
\qquad
Q_0=\frac12\,Q_1,
\qquad
Q_1 = \begin{bmatrix}
1 &     &     \\
1 & \ddots   &     \\
  & \ddots   & 1   \\
  &    & 1
\end{bmatrix} \in\mathbb{R}^{2^\ell\times m_\ell},
\]
where $Q_1\in\mathbb{R}^{2^\ell\times m_\ell}$ is the bidiagonal matrix with $1$ on the main diagonal and the first subdiagonal.
In particular, $Q_1$ has at most two nonzeros per row and per column and satisfies
$\max_{i,j}|(Q_1)_{ij}|\le 1$.

By Lemma~\ref{lem:sparse_access_be_rect},  matrix $Q_1$ admits an exact $(2,0)$-generalized block-encoding.
Scaling gives an exact $(1,0)$-generalized block-encoding of $Q_0=\frac12 Q_1$.
The identity block $I_{m_\ell\times m_\ell}$ admits an exact $(1,0)$-generalized block-encoding with respect to $(\Pi_{\mathrm{in}},\Pi_{\mathrm{in}})$.
By Remark~\ref{remark:Vertical stacking}, the vertical concatenation
$\bigl[\begin{smallmatrix}I_{m_\ell\times m_\ell}\\ Q_0\end{smallmatrix}\bigr]$ admits an exact generalized block-encoding with normalization
$\sqrt{1^2+1^2}=\sqrt{2}$. Hence there exists a unitary $U_{\mathrm{re}}$ such that
\[
\Pi_{\mathrm{out}}\,U_{\mathrm{re}}\,\Pi_{\mathrm{in}}^\dagger=\frac{1}{\sqrt{2}}\,P_{\mathrm{reorder}}.
\]

Finally, define $U_P:=U_\pi^\dagger U_{\mathrm{re}}$. Since $U_\pi$ maps $\mathcal{H}_{\mathrm{out}}$ onto itself, the output
projector satisfies $\Pi_{\mathrm{out}}U_\pi^\dagger=U_\pi^\dagger\Pi_{\mathrm{out}}$. Therefore,
\[
\Pi_{\mathrm{out}}\,U_P\,\Pi_{\mathrm{in}}^\dagger
=
\Pi_{\mathrm{out}}\,U_\pi^\dagger U_{\mathrm{re}}\,\Pi_{\mathrm{in}}^\dagger
=
U_\pi^\dagger\Big(\Pi_{\mathrm{out}}\,U_{\mathrm{re}}\,\Pi_{\mathrm{in}}^\dagger\Big)
=
\frac{1}{\sqrt{2}}\,U_\pi^\dagger P_{\mathrm{reorder}}
=
\frac{1}{\sqrt{2}}\,P.
\]
\end{proof}

\begin{lemma}\label{lem:BE_E_tensor_T}
	Let $\ell<L$ and $r:=2^{L-\ell}$. Let $E_{\ell \to L}^{(k)}$ be defined by \eqref{eq:E-tensor} and let
	\[
	E_{\ell \to L}=\mathrm{diag}\bigl(E_{\ell \to L}^{(1)},\cdots,E_{\ell \to L}^{(d)}\bigr),
	\qquad
	T_{\ell \to L}:=r^{-d/2}E_{\ell \to L}.
	\]
	Assume that the one-step 1D prolongation $P_{\ell\to \ell+1}^{1d}$ admits an exact generalized
	$(\sqrt{2},0)$-block-encoding in the sense of Definition~\ref{def:gen_blockencoding}.
	Then $T_{\ell \to L}$ admits an exact generalized block-encoding with normalization
	\[
	\alpha_T=\sqrt{d}.
	\]
\end{lemma}

\begin{proof}
	We combine generalized block-encodings using Lemma~\ref{lem:gen_arithmeticBE}.

	\textbf{Step 1: Block-encoding of $E_{\ell \to L}^{(1d)}$.}
	Recall $E_{\ell \to L}^{(1d)}=I_{\ell}\otimes \mathbf{1}_r$.
	The columns of $E_{\ell \to L}^{(1d)}$ are mutually orthogonal and each has Euclidean norm $\sqrt r$.
	Hence
	$\|E_{\ell \to L}^{(1d)}\|=\sqrt r.$
	Moreover, $E_{\ell \to L}^{(1d)}$ admits an exact generalized $(\sqrt r,0)$-block-encoding:
	for each column index $q$, prepare the normalized column state
	$\frac{1}{\sqrt r}\sum_{t=0}^{r-1}|rq+t\rangle$ on the output register, and use the standard
	row/column state-preparation construction for projected unitary encodings \cite{gilyen2019quantum}.
	Therefore,
	\begin{equation}\label{eq:norm_E_1d}
	\gamma(E_{\ell \to L}^{(1d)})=\sqrt r.
	\end{equation}
	
	\textbf{Step 2: Block-encoding of $P_\ell^{1d}$.}
	Let $P_\ell^{1d}$ denote the multistep 1D prolongation from level $\ell$ to level $L$,
	\[
	P_\ell^{1d}=P_{L-1\to L}^{1d}\cdots P_{\ell\to \ell+1}^{1d}.
	\]
	By Lemma~\ref{lem:BE_P_one_step_sqrt2}, each one-step prolongation admits an exact generalized
	$(\sqrt2,0)$-block-encoding. Repeatedly applying Lemma~\ref{lem:gen_arithmeticBE}(2) yields an exact generalized
	block-encoding of $P_\ell^{1d}$ with normalization
	\begin{equation}\label{eq:alpha_Pell_1d_gen}
		\gamma(P_\ell^{1d})=(\sqrt2)^{L-\ell}=\sqrt{2^{L-\ell}}=\sqrt r.
	\end{equation}
	
	\textbf{Step 3: Block-encoding of each tensor block $E_{\ell \to L}^{(k)}$.}
	By \eqref{eq:E-tensor}, for each $k\in\{1,\cdots,d\}$,
	\[
	E_{\ell \to L}^{(k)} =
	\Big( \bigotimes_{m=1}^{k-1} P_\ell^{1d} \Big) \otimes E_{\ell \to L}^{(1d)} \otimes
	\Big( \bigotimes_{m=k+1}^{d} P_\ell^{1d} \Big).
	\]
	Using Lemma~\ref{lem:gen_arithmeticBE}(3) repeatedly together with \eqref{eq:norm_E_1d} and
	\eqref{eq:alpha_Pell_1d_gen}, we obtain an exact generalized block-encoding of $E_{\ell \to L}^{(k)}$ with normalization
	\begin{equation}\label{eq:alpha_Ek_gen}
		\gamma_k
		=(\gamma(P_\ell^{1d}))^{d-1}\cdot \gamma(E_{\ell \to L}^{(1d)})
		=(\sqrt r)^{d-1}\cdot \sqrt r
		=r^{d/2}.
	\end{equation}
	
	\textbf{Step 4: Block-diagonal assembly.}
	Let $U_k$ be an exact generalized $(\gamma_k,0)$-block-encoding of $E_{\ell \to L}^{(k)}$ for $k=1,\cdots,d$,
	with respect to input/output projectors $(\Pi_{\mathrm{in}}^{(k)},\Pi_{\mathrm{out}}^{(k)})$.
	Prepare on a selector register $\mathsf{K}$ the state
	\[
	|\beta\rangle := \frac{1}{\gamma_E}\sum_{k=1}^d \gamma_k\,|k\rangle,
	\qquad
	\gamma_E:=\sqrt{\sum_{k=1}^d \gamma_k^2} = \sqrt{d}\, r^{d/2}.
	\]
	Define the unitary
	\[
	U_E := (U_{\beta}^\dagger\otimes I)\Bigl(\sum_{k=1}^d |k\rangle\langle k|_{\mathsf{K}}\otimes U_k\Bigr)(U_{\beta}\otimes I),
	\]
	where $U_{\beta}$ is any unitary satisfying $U_{\beta}|0\rangle=|\beta\rangle$.
	Let $\Pi_{\mathrm{in}}:=\bigoplus_{k=1}^d \Pi_{\mathrm{in}}^{(k)}$ and $\Pi_{\mathrm{out}}:=\bigoplus_{k=1}^d \Pi_{\mathrm{out}}^{(k)}$.
	Then $U_E$ is an exact generalized $(\gamma_E,0)$-block-encoding of
	\[
	E_{\ell \to L}=\mathrm{diag}\bigl(E_{\ell \to L}^{(1)},\cdots,E_{\ell \to L}^{(d)}\bigr),
	\qquad\text{i.e.,}\qquad
	E_{\ell \to L}=\gamma_E\,\Pi_{\mathrm{out}} U_E \Pi_{\mathrm{in}}^\dagger.
	\]
	Consequently, since $T_{\ell \to L}=r^{-d/2}E_{\ell \to L}$, the same unitary $U_E$ together with the same projectors
	$\Pi_{\mathrm{in}},\Pi_{\mathrm{out}}$ provides an exact generalized block-encoding of $T_{\ell \to L}$ with normalization
	$\alpha_T=r^{-d/2}\gamma_E=\sqrt d$.
\end{proof}
	
	\begin{lemma}[Block-encoding of $C_LS_\ell$]\label{lem:BE_CLSell}
		Assume that $T_{\ell \to L}$ admits an exact generalized $(\sqrt{d},0)$-block-encoding as established in
		Lemma~\ref{lem:BE_E_tensor_T}. Then $C_LS_\ell$ admits an exact generalized block-encoding with normalization
		$\alpha_{\ell}= \mathcal{O}(d)$.
	\end{lemma}
	
	\begin{proof}
		Throughout we use the generalized block-encoding calculus in Lemma~\ref{lem:gen_arithmeticBE}.
		
		Let $D_\ell^{1d}$ be the one-dimensional difference matrix defined in \eqref{eq:def DL}.
		By inspection of \eqref{eq:def DL}, $D_\ell^{1d}$ is $s_r$-row-sparse and $s_c$-column-sparse with $s_r=s_c=2$ and $\max_{i,j}|(D_\ell^{1d})_{ij}|\le 1$.
		By Lemma~\ref{lem:sparse_access_be_rect}, $D_\ell^{1d}$ admits an exact generalized $(\sqrt{s_rs_c},0)=(2,0)$-block-encoding, i.e.
		\begin{equation}\label{eq:gamma_D1d}
			\gamma (D_\ell^{1d})\le 2.
		\end{equation}
		
		For each $k\in\{1,\cdots,d\}$, let $D_\ell^{(k)}$ be the $k$-th directional difference operator defined in \eqref{eq:def DL}.
		Since $D_\ell^{(k)}$ is obtained from $D_\ell^{1d}$ by tensoring with identities, Lemma~\ref{lem:gen_arithmeticBE}(3) implies
		\begin{equation}\label{eq:gamma_Dk}
			\gamma(D_\ell^{(k)})=\gamma(D_\ell^{1d})\le 2.
		\end{equation}
		Let $D_\ell$ be the stacked derivative operator defined in \eqref{eq:DL_stack}.
		Applying the multi-block vertical stacking rule in Remark~\ref{remark:Vertical stacking}, we obtain an exact generalized
		block-encoding of $D_\ell$ with normalization
		\begin{equation}\label{eq:gamma_Dell}
			\gamma(D_\ell)=\sqrt{\sum_{k=1}^d \gamma(D_\ell^{(k)})^2}\le \sqrt{d}\cdot 2 = 2\sqrt{d}.
		\end{equation}

		Using the identity
		$C_LS_\ell = T_{\ell \to L}\,D_\ell,$ in Lemma \ref{thm:CS=TD},
		together with Lemma~\ref{lem:gen_arithmeticBE}(2), we obtain an exact generalized block-encoding of
		$C_LS_\ell$ with normalization
		\[
		\alpha_\ell \;=\;\gamma(T_{\ell \to L})\,\gamma(D_\ell)
		\;\le\; (\sqrt{d})\cdot (2\sqrt{d})
		\;=\; 2d.
		\]
		Therefore $\alpha_\ell=\mathcal{O}(d)$, as claimed.
	\end{proof}

\begin{theorem}\label{thm:BE_STAS}
	Let $A=C_L^\top C_L$ and let $S=[S_0, S_1, \cdots, S_L]$ be the BPX scaling operator.  The preconditioned matrix $A_S=(C_LS)^\top(C_LS)$ admits an exact (square) block-encoding with normalization
		\[
		\alpha_{A_S}:=\alpha_{CS}^2=\mathcal{O}\!\bigl(d^2(L+1)\bigr).
		\]
\end{theorem}

\begin{proof}
	
	By Lemma~\ref{lem:BE_CLSell}, for each $\ell\in\{0,\cdots,L\}$ there exists an exact generalized block-encoding $U_\ell$
	of $C_LS_\ell$ with normalization $\alpha_\ell=\mathcal{O}(d)$ such that
	\[
	C_LS_\ell=\alpha_\ell\,\Pi_{\mathrm{out}}^{(\ell)}\,U_\ell\,\bigl(\Pi_{\mathrm{in}}^{(\ell)}\bigr)^\dagger.
	\]
	
	Define the horizontal concatenation
	$
	C_LS=[C_LS_0,C_LS_1,\cdots, C_LS_L].
	$
	Lemma~\ref{lem:gen_arithmeticBE}(4) and its multi-block extension yield an exact generalized block-encoding $U_{CS}$ of $C_LS$
	with normalization
	\[
	\alpha_{CS}=\sqrt{\sum_{\ell=0}^L \alpha_\ell^2} =\mathcal{O}\!\bigl(d\sqrt{L+1}\bigr).
	\]
	
	Lemma~\ref{lem:gen_arithmeticBE}(1) implies that $U_{CS}^\dagger$ is an exact generalized block-encoding of $(C_LS)^\top$
	with the same normalization $\alpha_{CS}$.
	Applying Lemma~\ref{lem:gen_arithmeticBE}(2) to $(C_LS)^\top$ and $C_LS$ yields an exact block-encoding of
	\[
	A_S=(C_LS)^\top(C_LS)
	\]
	with normalization $\alpha_{A_S}=\alpha_{CS}^2$.
	Substituting the bound on $\alpha_{CS}$ gives $\alpha_{A_S}=\mathcal{O}\!\bigl(d^2(L+1)\bigr)$.
\end{proof}

}

\section{Implementation of the Schr\"odingerization based method} \label{sec:implementation}

In this section, we present the implementation of the Hamiltonian simulation \eqref{eq:hamiltonian} arising from the BPX
preconditioning framework. This includes the preparation of the initial quantum state, the realization of the unitary
evolution $\mathcal{U}(T)=e^{-iHT}$, and the computation of a linear quantity.

{\subsection{Preparation of the input state }}\label{subsec:prep_STb}

We begin by preparing the initial quantum state $\bm{W}_h(0)$ corresponding to \eqref{eq:hamiltonian}, where
\[
\bm{W}_h(0) =\bm\psi\otimes \bm{z}_f(0),\qquad
\bm\psi=\bigl[\psi(p_0),\cdots,\psi(p_{N_p-1})\bigr]^\top,\qquad
\bm{z}_f(0)=\ket{1}\otimes \bigl(T\,{\bm b}_S\bigr),
\]
where $\bm{b}_S = S^{\top}\bm{b}$.
We assume the following state-preparation unitaries:
\begin{itemize}
  \item Let $\mathsf{P}$ be an $n_p$-qubit register and let $O_\psi$ be a unitary such that
\begin{equation}\label{eq:Opsicorrect}
	O_{\psi}\bigl(|0^{n_p}\rangle_{\mathsf{P}}\bigr)=|\psi\rangle_{\mathsf{P}}
	:=\frac{1}{C_e}\sum_{k=0}^{N_p-1}\psi(p_k)\,|k\rangle_{\mathsf{P}},
	\qquad
	C_e=\|\bm\psi\|.
\end{equation}
  \item Let $\mathsf{B}$ be an $n_b$-qubit register and let $O_{Sb}$ be a unitary such that
\begin{equation}\label{eq:OSbcorrect}
	O_{Sb}\bigl(|0^{n_b}\rangle_{\mathsf{B}}\bigr)=|b_S\rangle_{\mathsf{B}}
	:=\frac{1}{\|{\bm b}_S\|}\sum_{j}({\bm b}_S)_j\,|j\rangle_{\mathsf{B}}.
\end{equation}
\end{itemize}

Let $\mathsf{a}$ be an $a$-qubit ancilla register and let $\mathsf{f}$ be a one-qubit flag register. Using the above two state-preparation oracles, we can prepare $\bm{w}_h(0)$ by defining $O_{\mathrm{prep}}$ on $\mathsf{a}\otimes\mathsf{P}\otimes\mathsf{f}\otimes\mathsf{B}$ as
\begin{equation}\label{eq:Oprep_correct}
	O_{\mathrm{prep}}\Bigl(|0^{a}\rangle_{\mathsf{a}}\otimes|0^{n_p}\rangle_{\mathsf{P}}\otimes (| 0\rangle_{\mathsf{f}}\otimes|0^{n_b}\rangle_{\mathsf{B}})\Bigr)
	=
	|0^{a}\rangle_{\mathsf{a}}\otimes|\psi\rangle_{\mathsf{P}}\otimes (|1\rangle_{\mathsf{f}}\otimes|b_S\rangle_{\mathsf{B}}).
\end{equation}
where $|b_S\rangle_{\mathsf{B}}$ is the normalization state of $\bm{b}_S$.
In amplitude-encoding form, the right-hand side equals the normalized state associated with $\bm{w}_h(0)$ up to the known
scaling factors $C_e$ and $T\|\bm{b}_S\|$. Hence we may write
\[
O_{\mathrm{prep}}\Bigl(|0^{a}\rangle_{\mathsf{a}}\otimes|0^{n_p}\rangle_{\mathsf{P}}\otimes|0\rangle_{\mathsf{f}}\otimes|0^{n_b}\rangle_{\mathsf{B}}\Bigr)
=
|0^{a}\rangle_{\mathsf{a}}\otimes \frac{1}{\eta_0}\,\bm{W}_h(0)\;+\;|\bot\rangle,
\qquad
\eta_0:=C_e\,T\,\|{\bm b}_S\|,
\]
where $|\bot\rangle$ is a state orthogonal to $|0^{a}\rangle_{\mathsf{a}}$ on the ancilla register.
For simplicity, we denote the above procedure as
\begin{equation}\label{Oprep}
	|0^{a}\rangle_{\mathsf{a}}\otimes|0^{n_p}\rangle_{\mathsf{P}}\otimes|0\rangle_{\mathsf{f}}\otimes|0^{n_b}\rangle_{\mathsf{B}}
	\ \xrightarrow{\ O_{\mathrm{prep}}\ }\
	|0^{a}\rangle_{\mathsf{a}}\otimes \frac{1}{\eta_0}\,\bm{W}_h(0)\;+\;|\bot\rangle.
\end{equation}

\subsection{Implementation of the evolution operator  $\mathcal{U}(T)=e^{-iHT}$}

One can express the evolution operator $\mathcal{U}(T)=e^{-\mathrm{i}HT}$ as a select unitary acting on the Fourier-mode
register $\mathsf{P}$:
\begin{equation}\label{eq:select_Vk}
	\mathcal{U}(T)
	=\sum_{k=0}^{N_p-1} |k\rangle\langle k|_{\mathsf{P}}\otimes
	\exp\bigl(-\mathrm{i}(\mu_k H_1-H_2)T\bigr)
	=: \sum_{k=0}^{N_p-1} |k\rangle\langle k|_{\mathsf{P}}\otimes V_k(T),
\end{equation}
where $H_{\mu_k}:=\mu_k H_1-H_2$ and $V_k(T)=e^{-\mathrm{i}H_{\mu_k}T}$.

According to Theorem \ref{thm:BE_STAS},  $H_1$ and $H_2$ admit exact block-encodings $U_{H_1}$ and $U_{H_2}$
on an ancilla register $\mathsf{a}$ and a system register
such that
\begin{equation}\label{eq:BE_Hi}
	\bigl(\langle 0^{a}\!|_{\mathsf{a}}\otimes I\bigr)\,U_{H_i}\,\bigl(|0^{a}\rangle_{\mathsf{a}}\otimes I\bigr)
	=\frac{H_i}{\alpha_i^H},
	\qquad i\in\{1,2\},
\end{equation}
with normalization factors
\[
\alpha_1^H= \alpha_{A_S}+\frac{1}{2T},
\qquad
\alpha_2^H=\frac{1}{T}.
\]
The associated gate complexity is $\mathcal{O}(T_{A_S})$ up to a constant overhead, where $T_{A_S}$ denotes the cost of implementing the block-encoding of $A_S$.
%
Let
\begin{equation}\label{alphaf}
\mu_{\max}:=\max_{0\le k\le N_p-1}|\mu_k|,
\qquad
\alpha_f:=\alpha_1^H\,\mu_{\max}+\alpha_2^H.
\end{equation}

Following the construction in Section~4.2.1 of \cite{ACL2023LCH2}, there exists a unitary oracle
$\mathrm{HAM\!-\!T}_{H_\mu}$ such that
\begin{equation}\label{eq:HAMT}
	\bigl(\langle 0^{a}\!|_{\mathsf{a}}\otimes I\bigr)\,
	\mathrm{HAM\!-\!T}_{H_\mu}\,
	\bigl(|0^{a}\rangle_{\mathsf{a}}\otimes I\bigr)
	=\sum_{k=0}^{N_p-1} |k\rangle\langle k|_{\mathsf{P}}\otimes \frac{H_{\mu_k}}{\alpha_f}.
\end{equation}
This oracle uses $\mathcal{O}(1)$ queries to $U_{H_1}$ and $U_{H_2}$ (and their inverses), and $\mathrm{poly}(a)$
additional elementary gates.
With the block-encoding \eqref{eq:HAMT}, we apply quantum singular value transformation to implement a select unitary
\begin{equation}\label{eq:SEL_U}
	\mathrm{SEL}_{\mathcal{U}}:=\sum_{k=0}^{N_p-1} |k\rangle\langle k|_{\mathsf{P}}\otimes V_k^a(T),
\end{equation}
where each $V_k^a(T)$ approximates $V_k(T)=e^{-\mathrm{i}H_{\mu_k}T}$ to the prescribed accuracy of $\delta$.

We refer to \cite{ACL2023LCH2} and the implementation details summarized in \cite{JLMY2025qSmooth} for the explicit QSVT
polynomial construction and the resulting complexity bounds. To this end, we introduce the following results.

\begin{lemma}[Theorem 2.1 of \cite{JLMY2025qSmooth}] \label{lem:errw}
Let $\bb{w}(t,p)$ be the exact solution to \eqref{eq:shcro w}, and let $\bm{W}_h(t)$ denote the solution of the discrete problem \eqref{eq:hamiltonian}. Assume that $\psi \in H^r(\mathbb{R})$ and decays exponentially on $\mathbb{R}$.
Suppose the mesh size $\triangle p$ satisfies
\begin{equation}\label{mumax}
 (\triangle p)^{-1} \simeq \mu_{\max} \simeq \pi (1/\epsilon)^{1/r} \|\psi^{(r)}\|_{L^2((-R,R))}^{1/r},
\end{equation}
where  $R$ is chosen according to \eqref{eq: L,R,criterion}.
Then the following error estimate holds:
\begin{equation}\label{errw}
\| \bb{w}(T,p) - \bb{w}_h(T,p) \|_{L^2((-R,R))} \lesssim  \epsilon \|\bm{z}_f(0)\|,
\end{equation}
where $\bb{w}_h$ is the continuous reconstruction of $\bm{W}_h$, given by
\[
\bb{w}_h(t,p) = \sum_{l=0}^{N_p-1} \tilde{\bb{w}}_{l,h}(t)\,\phi_l(p),
\qquad
\tilde{\bb{w}}_{l,h}(t) = \frac{1}{N_p} \sum_{k=0}^{N_p-1} \big((\bra{k}\otimes I)\bm{W}_h\big) \,
\e^{ - \i \mu_l (p_k+R)}.
\]
\end{lemma}

\begin{lemma}[{ Implementation of $\mathcal{U}(T)$}]\label{lem:U_sim}
Under the condition of Lemma \ref{lem:errw}, we assume that the function $\psi\in H^r(\mathbb{R})$ in the initial data of \eqref{initw} decays exponentially on $\mathbb{R}$ and satisfies
\[\|\psi^{(r)}\|_{L^2((-R,R))}^{1/r} \le C r, \quad \beta \in (0,1], \]
where $r \simeq \log(1/\epsilon)$ and $C$ is a constant independent of $\epsilon$.
Let $\delta \in(0,1)$ be a target accuracy parameter for each $V_k^a(T)$ in \eqref{eq:SEL_U}.
There exists a quantum algorithm $V^a$ that, given the state-preparation
	oracle $O_{\mathrm{prep}}$ in \eqref{Oprep}, produces
	\[
	|0\rangle_{\mathsf{a}}\,|0\rangle_{\mathsf{b}}
	\ \xrightarrow{\ V^a\ }\
	\frac{1}{\eta_0}\,|0\rangle_{\mathsf{a}}\otimes \bm{W}_h^a(T) + |\bot\rangle,
	\]
	where $\bm{W}_h^a(T)=\mathcal{U}^a(T)\bm{W}_h(0)$ is an approximation of $\bm{W}_h(T)$ in \eqref{eq:hamiltonian}, satisfying
	\[
	\|\bm{W}(T)-\bm{W}_h^a(T)\|\ \lesssim\  \mu_{\max}^{1/2} ( \delta + \epsilon) \|\bm{z}_f(0)\|,
	\]
with $\bm{W}(t) = \sum_{ki} w_i(t,p_k) \ket{k,i}$ being the exact solution.
	Moreover, $V^a$ uses
\begin{equation}\label{complexityU}
	\mathcal{O} \Big(\alpha_f\,T\ \log\!\frac{\|\bm{z}_f(0)\|}{\delta\,\|\bm{z}(T)\|}\Big)
\end{equation}
	queries to $\mathrm{HAM\!-\!T}_{H_\mu}$, and $\mathcal{O}(1)$ queries to $O_{\mathrm{prep}}$, where $\alpha_f$ is defined in \eqref{alphaf}.
\end{lemma}
\begin{proof}
See the proof of Theorem 2.2 in \cite{JLMY2025qSmooth}.
\end{proof}

{
\subsection{Time complexity of computing a linear quantity}\label{subsec:measurement}

The Schr\"odingerization-based evolution produces the extended state $\bm{W}_h(T)$, whose Fourier-mode components are indexed
by $p_k$. In particular, the coefficient vector of interest $\bm{z}(T)$ is not returned directly,
where $\bm{z}(T)$ is the solution of the preconditioned system \eqref{eq:ODE square root}.
Instead, the algorithm
outputs a frequency-weighted quantity $\bm{z}_f(T)$ supported on the admissible modes $p_k\ge p_\Diamond$, and $\bm{z}(T)$ is
recovered from $\bm{z}_f(T)$ by undoing the exponential weight.

 { Since \(\bm x \approx \bm u(T) = S \bm z(T)\),  we focus on linear observables of the form \(\bm c^\top \bm u(T)\), and more generally, expectation-type measurements that reduce to inner products in the BPX coefficient space.  In particular,
$$
\bm c^\top \bm u(T) = (S^\top \bm c)^\top \bm z(T).
$$
Such linear quantities are natural in PDE-based models. They represent, for example, pointwise evaluation at a sensor location, spatial averages over a region, fluxes across a boundary segment, or the output of a detector with a prescribed response profile encoded by \(\bm c\). This measurement model is aligned with the Schr\"odingerization recovery step, so the target quantity can be estimated from overlaps and expectation values without explicitly outputting the full solution vector.}

\subsubsection{Coherent preparation of the target state}
Let $\mathsf{P}$ be the $n_p$-qubit mode register with computational basis
$\{|k\rangle\}_{k=0}^{N_p-1}$, and let $\mathsf{Z}$ be the register holding the BPX coefficient state.
For admissible modes $p_k\ge p_\Diamond$, we have the relation
\begin{equation}\label{eq:zf_w_relation}
	\bm z_f(T;k)=e^{p_k}\,\bm w(T;k),
	\qquad p_k\ge p_\Diamond,
\end{equation}
where $\bm w(T;k)$ denotes the corresponding component of $\bm W_h(T)$.
Thus, the target coefficient vector $\bm z(T)$ is recovered from the simulated extended state by selecting the admissible branch
and undoing the exponential weight.

As analyzed in \cite{JLMY2025qSmooth} (see Eq.~(2.18) there), the success probability for obtaining the normalized target state
\[
|\bm z(T)\rangle:=\frac{\bm z(T)}{\|\bm z(T)\|}
\]
from the output of the simulation is
\begin{equation}\label{eq:Pz}
	\mathrm P_{\mathrm r}(\bm z)
	=
	\frac{C_{e0}^2}{C_e^2}\,
	\frac{\|\bm z(T)\|^2}{\|\bm z_f(0)\|^2}
	=
	\frac{C_{e0}^2}{C_e^2}\,
	\frac{\|\bm z(T)\|^2}{(T\|{\bm b}_S\|)^2},
\end{equation}
where
\[
C_{e0}=\Bigl(\sum_{p_k\ge p_\Diamond}\psi(p_k)^2\Bigr)^{1/2},
\qquad
C_e=\Bigl(\sum_{k=0}^{N_p-1}\psi(p_k)^2\Bigr)^{1/2},
\qquad
p_\Diamond=\frac12.
\]
Moreover, the ratio $C_{e0}^2/C_e^2$ is bounded above and below by positive constants independent of the discretization parameters. In particular,
\[
\frac{C_{e0}^2}{C_e^2}=\Theta(1).
\]
For example, for $\psi(p)=e^{-|p|}$, this ratio is approximately $1/2$ when the discretization in $p$ is sufficiently fine.

We emphasize that no intermediate measurement is performed in this recovery procedure. Instead, the preparation of 
$|\bm z(T)\rangle$ is kept fully coherent and appears as the success branch of the overall state-preparation circuit.
The factor
$\mathrm P_{\mathrm r}(\bm z)$ will be accounted for later through amplitude amplification/amplitude estimation in the final
complexity bound.

\begin{remark}
If no preconditioner is used, then $S^\top=S=I$, so that $A_S=A$ and the target state $|z(T)\rangle$ coincides with the solution state $|u(T)\rangle$, which approximates $|x\rangle$. In this case, the block-encoding normalization factor becomes $\alpha_f=\|A\|\mu_{\max}=\lambda_{\max}(A)\mu_{\max}$. Combining Lemma \ref{lem:U_sim} with the success-probability factor $\mathrm P_{\mathrm r}(z)$ and the choice
\[
T=\Theta\Big(\frac{1}{\lambda_{\min}(A)}\log\frac1\varepsilon\Big),
\]
the overall oracle complexity reduces to
\[
\mathcal O\Big(\kappa_A^2\,\mu_{\max}\,\mathrm{polylog}\Big(\frac{\kappa_A}{\varepsilon}\Big)\Big).
\]
Moreover, if $\psi$ is sufficiently smooth so that $\mu_{\max}=\mathcal O(\log(1/\varepsilon))$, this further simplifies to
\[
\mathcal O\Big(\kappa_A^2\,\mathrm{polylog}\Big(\frac{\kappa_A}{\varepsilon}\Big)\Big).
\]
Thus, in the absence of preconditioning, our framework recovers the usual $\kappa_A^2$-type dependence up to polylogarithmic factors.
\end{remark}

\subsubsection{Linear measurements in the physical space.}

We do not output $\bm{u}(T)$ in \eqref{eq:ODE} explicitly. As analyzed in \cite{DP24}, the symmetric preconditioning makes it impossible to measure quantities of interest of the form $\langle u |M| u\rangle$. Instead, we focus on the absolute value of a linear quantity
\[
\Upsilon:=|\bm c^\top \bm u(T)|.
\]
Since $\bm{u}(T)=S\,\bm{z}(T)$, any linear quantities
$\bm{c}^\top \bm{u}(T)$ can be written as
\begin{equation}\label{eq:ctu_reduce}
	\bm{c}^\top \bm{u}(T) = \bm{c}^\top S\,\bm{z}(T) = (S^\top \bm{c})^\top \bm{z}(T) = \bm{c}_S^\top \bm{z}(T).
\end{equation}
Thus, it suffices to estimate the inner product between $\bm{z}(T)$ and the scaled measurement vector
$\bm{c}_S:=S^\top \bm{c}$.

In line with our algorithmic focus, we assume the availability of a state-preparation oracle for $\bm{c}_S$:
there exists a unitary $O_{Sc}$ on a register $\mathsf{C}$ such that
\begin{equation}\label{eq:OSc_correct}
	O_{Sc}\bigl(|0^{n_c}\rangle_{\mathsf{C}}\bigr)=|c_S\rangle_{\mathsf{C}}
	:=\frac{1}{\|\bm{c}_S\|}\sum_{j}(\bm{c}_S)_j\,|j\rangle_{\mathsf{C}}.
\end{equation}
On the other hand, the target state $|z(T)\rangle=\bm z(T)/\|\bm z(T)\|$ is not prepared directly from the initial state by a unitary acting only on the coefficient register. Rather, it is obtained coherently from the simulated extended state by extracting the component corresponding to the target mode. The corresponding success probability is precisely $\mathrm P_{\mathrm r}(\bm z)$ in \eqref{eq:Pz}. After amplitude amplification, this yields an effective state-preparation routine for $|z(T)\rangle$ with constant success probability, up to the stated approximation error.

For the absolute value of the linear quantity \eqref{eq:ctu_reduce}, in quantum setting, we need to consider the overlap
\[
\mu:=|\langle c_S\mid z(T)\rangle|,
\]
which can be estimated by a standard quantum inner-product estimation routine.
Since
\[
\Upsilon
=
|\bm c^\top \bm u(T)|
=
|\bm c_S^\top \bm z(T)|
=
\|\bm c_S\|\,\|\bm z(T)\|\,\mu,
\]
an approximation $\widetilde{\mu}$ of $\mu$ gives rise to the estimator
\[
\widetilde{\Upsilon}
:=
\|\bm c_S\|\,\|\bm z(T)\|\,\widetilde{\mu},
\]
which satisfies
\[
|\widetilde{\Upsilon}-\Upsilon|
\le
\|\bm c_S\|\,\|\bm z(T)\|\,|\widetilde{\mu}-\mu|.
\]
Hence, to achieve additive accuracy
\[
|\widetilde{\Upsilon}-\Upsilon|\le \mathrm{tol},
\]
it suffices to estimate $\mu$ by
\[
|\widetilde{\mu}-\mu|
\le
\frac{\mathrm{tol}}{\|\bm c_S\|\,\|\bm z(T)\|}.
\]
Therefore, it is enough to use a standard quantum inner-product estimation routine that approximates
\[
\mu=|\langle c_S\mid z(T)\rangle|
\]
directly to additive error
\begin{equation}\label{eq:epsprime_readout}
\varepsilon'
:=
\frac{\mathrm{tol}}{\|\bm c_S\|\,\|\bm z(T)\|}.
\end{equation}

By standard amplitude-estimation arguments, estimating $\mu$ to additive error $\varepsilon'$ requires
\[
\mathcal O\Big(\frac{1}{\varepsilon'}\Big)
=
\mathcal O\Big(
\frac{\|\bm c_S\|\,\|\bm z(T)\|}{\mathrm{tol}}
\Big)
\]
calls to the overlap-estimation procedure. Each such invocation uses a preparation of $|z(T)\rangle$. Since $|z(T)\rangle$ is obtained as the good branch of a coherent recovery procedure with success weight $\mathrm P_{\mathrm r}(\bm z)$, promoting this branch to constant success probability by amplitude amplification introduces an additional factor $\mathcal O(\mathrm P_{\mathrm r}(\bm z)^{-1/2})$. Hence the total number of underlying calls to the coherent recovery routine and the inner-product estimation procedure is
\begin{equation}\label{eq:g_total_new}
g
=
\mathcal O\Big(
\mathrm P_{\mathrm r}(\bm z)^{-1/2}\,\varepsilon'^{-1}
\Big)
=
\mathcal O\Big(
\frac{T\|{\bm b}_S\|\,\|\bm c_S\|}{\mathrm{tol}}
\Big),
\end{equation}
where we used \eqref{eq:Pz} and absorbed the factor $C_e/C_{e0}$ into the implicit constant.

\begin{theorem}
Assume the conditions of Lemma~\ref{lem:errw} and Lemma~\ref{lem:U_sim}. Let $\varepsilon\in(0,1)$ be the target accuracy
parameter and let $\mathrm{tol}=\mathcal{O}(\varepsilon)$ be the target precision for
$\Upsilon=|\bm c^\top \bm u(T)|$. There exists a quantum algorithm that outputs an estimate $\widetilde{\Upsilon}$ satisfying
$|\widetilde{\Upsilon}-\Upsilon|\le \mathrm{tol}$.

The algorithm consists of a coherent recovery stage that produces $|z(T)\rangle=\bm z(T)/\|\bm z(T)\|$ and a readout stage
that estimates $\mu:=|\langle c_S|z(T)\rangle|$ with $|c_S\rangle=\bm c_S/\|\bm c_S\|$, $\bm c_S=S^\top\bm c$.

\begin{itemize}
\item One invocation of the underlying simulation-and-recovery subroutine uses
\[
\mathcal{O}\Big(
\mathrm{poly}(d)\,
\log\frac{1}{\varepsilon}\,
\log\!\frac{\|\bm c_S\|}{\varepsilon\,\|\bm z(T)\|}
\Big)
\]
queries to the block-encoding oracle $U_{A_S}$ of $S^\top A S$, and $\mathcal{O}(1)$ queries to $O_{\mathrm{prep}}$ in \eqref{Oprep}.

\item The readout stage uses $\mathcal{O}(1)$ queries to $O_{Sc}$ per repetition. To achieve additive error
$|\widetilde{\Upsilon}-\Upsilon|\le \mathrm{tol}=\mathcal{O}(\varepsilon)$, one uses
\[
g=\mathcal O\!\left(\frac{\|{\bm b}_S\|\,\|\bm c_S\|}{\varepsilon}\right)
\]
invocations; consequently, the total numbers of queries to $O_{Sc}$ and $O_{\mathrm{prep}}$ are both $\mathcal O(g)$.
\end{itemize}

Therefore, the total number of queries to $U_{A_S}$ is
\[
\mathcal{O}\Big(
g\,
\mathrm{poly}(d)\,
\log\frac{1}{\varepsilon}\,
\log\!\frac{\|\bm c_S\|}{\varepsilon \,\|\bm z(T)\|}
\Big).
\]
\end{theorem}
\begin{proof}
Let $\bb{W}_h^a(T)$ be the solution associated with $\mathcal{U}^a$, where $\mathcal{U}^a$ is the approximation of $\mathcal{U}$. According to \eqref{eq:recovery}, one has
\[\bm{z}(T) = \e^{p_k}(\bra{k}\otimes \bra{0} \otimes I)\bm{W}(T), \qquad
\bm{z}_h^a(T) = \e^{p_k}(\bra{k}\otimes \bra{0} \otimes I)\bm{W}^a_h(T)\]
for some $k \in I_{\Diamond}$, where $\bm{W}(t) = \sum_{ki} w_i(t,p_k) \ket{k,i}$. Here, we can choose $p_k = \mathcal{O}(1)$.  This gives
\begin{equation}\label{errdiff}
\|\bm{z}(T) - \bm{z}_h^a(T)\| \le \e^{p_k} \| \bb{W}(T) - \bb{W}_h^a(T) \|, \qquad  k \in I_{\Diamond}.
\end{equation}
Accordingly, we can define the linear quantity and its approximation by
\[\Upsilon = \bm{c}^\top \bm{u}(T) = \bm{c}_S^\top \bm{z}(T), \qquad
 \Upsilon_h^a = \bm{c}^\top \bm{u}_h^a(T) = \bm{c}_S^\top \bm{z}_h^a(T),\]
with the error given by
\[|\Upsilon-\Upsilon_h^a| = |\bm{c}_S^\top (\bm{z}(T)-\bm{z}_h^a(T)) | \le \|\bm{c}_S\| \|\bm{z}(T)-\bm{z}_h^a(T)\| \le \varepsilon, \]
where $\varepsilon$ is the desired error bound for the computation of the linear quantity.
Then, we need to bound the error between $\bb{z}(T)$ and $\bb{z}_h^a(T)$ such that
\[\|\bm{z}(T)-\bm{z}_h^a(T)\| \le \frac{\varepsilon}{\|\bm{c}_S\|}.\]
According to \eqref{errdiff} and Lemma \ref{lem:U_sim}, we can require
\[\| \bb{W}(T) - \bb{W}_h^a(T) \| \le \mu_{\max}^{1/2} ( \delta + \epsilon) \|\bm{z}_f(0)\| \le \e^{-p_k} \frac{\varepsilon}{\|\bm{c}_S\|} \lesssim  \frac{\varepsilon}{\|\bm{c}_S\|}.\]

When $r \simeq \log(1/\epsilon)$, there holds
\[\mu_{\max} \simeq  \|\psi^{(r)}\|_{L^2((-L,R))}^{1/r} \lesssim r \simeq \log\frac{1}{\epsilon}. \]
For this reason, we let
\[\Big(\log\frac{1}{\epsilon}\Big)^{1/2} \epsilon \|\bm{z}_f(0)\| \simeq \frac{\varepsilon}{2\|\bm{c}_S\|}, \qquad
\Big(\log\frac{1}{\epsilon}\Big)^{1/2}  \delta \|\bm{z}_f(0)\| \simeq \frac{\varepsilon}{2\|\bm{c}_S\|}.\]
From the first equation, 
\[\epsilon \simeq \frac{\epsilon'}{(\log(1/\epsilon'))^{1/2}}, \qquad \epsilon' =\frac{\varepsilon\|\bm{z}_f(0)\|}{\|\bm{c}_S\|}. \]
This yields
\[\mu_{\max}
 \lesssim \log \frac{(\log(1/\epsilon'))^{1/2}} {\epsilon'}
\simeq \log \frac{\|\bm{c}_S\|}{\varepsilon \|\bm{z}_f(0)\|}, \]
\[\frac{1}{\delta} \simeq \frac{\|\bm{c}_S\|}{\varepsilon \|\bm{z}_f(0)\|} \log \frac{\|\bm{c}_S\|}{\varepsilon \|\bm{z}_f(0)\|}.\]
Plugging the above quantities into \eqref{complexityU} gives
\[
\mathcal{O} \Big(\alpha_f T  \log \frac{\|\bm{z}_f(0)\|}{\delta\,\|\bm{z}(T)\|}\Big)
= \mathcal{O} \Big(\alpha_{A_S} T  \log \frac{\|\bm{c}_S\|}{\varepsilon \|\bm{z}(T)\|}\Big)
= \mathcal{O} \Big(\alpha_{A_S} T  \log \frac{\|\bm{c}_S\|}{\varepsilon \|\bm{z}(T)\|}\Big).\]
Since the error for the finite difference method is $\mathcal{O}(d h_L^2)$, we can deduce that $L = \mathcal{O}(\log(d/\varepsilon))$, yielding
\[\alpha_{A_S}=\mathcal{O}(d^2(L+1)) = \text{poly}(d) \log (1/\varepsilon).\]
According to  Theorem \ref{thm:ODESS} and $\lambda(BA) = \Theta(\text{poly}(d))$ from \eqref{eq:eig BA}, one has
\[ T =\Theta\Big(\frac{1}{\lambda_{\min}(BA)}\log\frac{1}{\varepsilon}\Big) =\Theta(\text{poly}\log\frac{1}{\varepsilon}).\]
The proof is finished by multiplying the repeated times $g$ shown in \eqref{eq:g_total_new}.
\end{proof}
    
}

\section{Conclusions}\label{sec:con}

In this paper, we have developed a quantum preconditioning framework for solving linear systems { arising from finite difference discretizations of the Poisson equation}, combining the Schr\"odingerization technique with the BPX multilevel preconditioner. The core idea is to construct a structure-aware block-encoding of the symmetrized preconditioned matrix $A_S = S^\top A S$ via a novel commuting identity $C_L S_\ell = T_{\ell \to L} D_\ell$. This construction avoids the unfavorable normalization scaling that would otherwise arise from naive multiplication of separate block-encodings, yielding an exact block-encoding of $A_S$ with normalization $\mathcal{O}(d^2(L+1))$, where $d$ is the spatial dimension and $L$ is the number of levels. Combined with Schr\"odingerization-based Hamiltonian simulation, the overall quantum algorithm achieves a query complexity of $\mathcal{O}\big(\mathrm{poly}(d)\,\varepsilon^{-1}\,\mathrm{polylog}(\varepsilon^{-1})\big)$ for estimating linear quantities of the solution to a given tolerance.

{ The proposed construction may be extended to finite element discretizations by appropriately modifying the operator $C_L$. Moreover, our treatment of the commuting identity offers a simplified alternative to the discontinuous Galerkin-based approach in \cite{DP24}, potentially streamlining the quantum realization of preconditioned systems in broader contexts. A detailed investigation of these extensions is left for future work.
This enhancement broadens the applicability of the block-encoding framework to hierarchical elliptic/fractional Laplacian systems via the BPX preconditioner \cite{BNWX23} and paves the way for future extensions, such as integrating Nodal Auxiliary Space Preconditioning \cite{HX07} for H(curl)-conforming electromagnetic systems, further demonstrating the versatility of Schr\"odingerization-based preconditioning for complex PDEs.}

\section*{Acknowledgments}
SJ and NL were supported by NSFC grant No. 12341104,  the Shanghai Pilot Program for Basic Research, the Shanghai Jiao Tong University 2030 Initiative and the Fundamental Research Funds for the Central Universities. SJ was also partially supported by the NSFC grant No. 92270001.
NL also acknowledges funding from the Science and Technology Program of Shanghai, China (21JC1402900), the Science and Technology Commission of Shanghai Municipality (STCSM) grant no. 24LZ1401200 (21JC1402900) and NSFC grant No.12471411.
SJ, NL and CM were supported by the  Science and Technology Innovation Key R\&D Program of Chongqing grant No. CSTB2024TIAD-STX0035.
CM was partially supported by NSFC grant No. 12501607, the Science and Technology Commission of Shanghai Municipality (No.22DZ2229014),
China Postdoctoral Science Foundation (No. 2023M732248) and Postdoctoral Innovative Talents
Support Program (No. BX20230219).
YY was supported by NSFC grant (No.\ 12301561), the Key Project of Scientific Research Project of Hunan Provincial Department of Education (No.\ 24A0100), the Science and Technology Innovation Program of Hunan Province (No.\ 2025RC3150) and the general program of Hunan Provincial Natural Science Foundation (No.\ 2026JJ50003).
YY was supported in part by the 111 Project (No.\ D23017), and Program for Science and Technology Innovative Research Team in Higher Educational Institutions of Hunan Province of China.

\newpage
\bibliographystyle{unsrt} 
\bibliography{references}

\end{document}